\documentclass[a4paper,11pt,reqno]{amsart}
\usepackage{amssymb,amsmath,mathrsfs,graphics,graphicx,mathptm,float,lscape,latexsym,pstricks}
\usepackage[T1]{fontenc}
\usepackage[english,francais]{babel}
\usepackage[all]{xy}
\theoremstyle{plain}
\newtheorem{thm}{Theorem}[section]
\newtheorem{nota}[thm]{Notations}
\newtheorem{pro}[thm]{Proposition}
\newtheorem{lem}[thm]{Lemma}
\newtheorem{cor}[thm]{Corollary}
\newtheorem{theoalph}{Theorem}

\newtheorem{proalph}[theoalph]{Proposition}

\theoremstyle{definition}
\newtheorem*{defi}{Definition}

\newtheorem{eg}[thm]{Example}

\newtheorem{rem}[thm]{Remark}

\def\og{\leavevmode\raise.3ex\hbox{$\scriptscriptstyle\langle\!\langle$~}}
\def\fg{\leavevmode\raise.3ex\hbox{~$\!\scriptscriptstyle\,\rangle\!\rangle$}}

\addtocounter{section}{0}             
\numberwithin{equation}{section}       

\usepackage[colorlinks, linktocpage, citecolor = blue, linkcolor = blue]{hyperref}

\begin{document}
\selectlanguage{english}
\title{On cubic birational maps of $\mathbb{P}^3_\mathbb{C}$}
\thanks{First author partially supported by ANR Grant "BirPol"  ANR-11-JS01-004-01.}

\author{Julie \textsc{D\'eserti}}
\email{julie.deserti@imj-prg.fr}

\author{Fr\'ed\'eric \textsc{Han}}
\email{frederic.han@imj-prg.fr}

\address{Univ Paris Diderot, Sorbonne Paris Cit\'e, Institut de Math\'ematiques de Jussieu-Paris Rive Gauche, UMR $7586$, CNRS, Sorbonne Universit\'es, UPMC Univ Paris $06$, F-$75013$ Paris, France.}

\maketitle

\begin{abstract}
We study the birational maps of $\mathbb{P}^3_\mathbb{C}$. More precisely we describe the irreducible components of the set of birational maps of bidegree $(3,3)$ (resp. $(3,4)$, resp. $(3,5)$). 
\end{abstract}

\section{Introduction}\label{Sec:intro}

The \textsc{Cremona} group, denoted $\mathrm{Bir}(\mathbb{P}^n_\mathbb{C})$, is the group of birational maps of $\mathbb{P}^n_\mathbb{C}$ into itself. If $n=2$ a lot of properties have been established (\emph{see} \cite{Cantat1, Deserti} for example). As far as we know the situation is much more different for $n\geq 3$ (\emph{see} \cite{Pan3, Cantat2} for example). If $\psi$ is an element of $\mathrm{Bir}(\mathbb{P}^2_\mathbb{C})$ then $\deg\psi=\deg\psi^{-1}$. It is not the case in higher dimensions; if $\psi$ belongs to $\mathrm{Bir}(\mathbb{P}^3_\mathbb{C})$ we only have the inequality $\deg\psi^{-1}\leq(\deg\psi)^2$ so one introduces the bidegree of $\psi$ as the pair $(\deg\psi,\deg\psi^{-1})$. For $n=2$, $\mathfrak{Bir}_d(\mathbb{P}^2_\mathbb{C})$ is the set of birational maps of the complex projective plane of degree $d$; for $n\geq 3$ denote by $\mathrm{Bir}_{d,d'}(\mathbb{P}^n_\mathbb{C})$ the set of elements of $\mathrm{Bir}(\mathbb{P}^n_\mathbb{C})$ of bidegree $(d,d')$, and by $\mathfrak{Bir}_d(\mathbb{P}^n_\mathbb{C})$ the union $\cup_{d'}\mathrm{Bir}_{d,d'}(\mathbb{P}^n_\mathbb{C})$.

In \cite{CerveauDeserti} the sets $\mathfrak{Bir}_2(\mathbb{P}^2_\mathbb{C})$, and $\mathfrak{Bir}_3(\mathbb{P}^2_\mathbb{C})$ are described: $\mathfrak{Bir}_2(\mathbb{P}^2_\mathbb{C})$ is smooth, and irreducible in the space of quadratic rational maps of the complex projective plane whereas $\mathfrak{Bir}_3(\mathbb{P}^2_\mathbb{C})$ is irreducible, and rationnally connected. Besides, $\mathfrak{Bir}_d(\mathbb{P}^2_\mathbb{C})$ is not irreducible as soon as $d>3$ (\emph{see}~\cite{BisiCalabriMella}). In \cite{Cremona} \textsc{Cremona} studies three types of generic elements of $\mathfrak{Bir}_2(\mathbb{P}^3_\mathbb{C})$. Then there were some articles on the subject, and finally a precise description of $\mathfrak{Bir}_2(\mathbb{P}^3_\mathbb{C})$; the left-right conjugacy is the following one
\[
\mathrm{PGL}(4;\mathbb{C})\times\mathrm{Bir}(\mathbb{P}^3_\mathbb{C})\times\mathrm{PGL}(4;\mathbb{C})\to\mathrm{Bir}(\mathbb{P}^3_\mathbb{C}),\quad\quad (A,\psi,B)\mapsto A\psi B^{-1}.
\]
\textsc{Pan}, \textsc{Ronga} and \textsc{Vust} give quadratic birational maps of $\mathbb{P}^3_\mathbb{C}$ up to left-right conjugacy, and show that there are only finitely many biclasses (\cite[Theorems 3.1.1, 3.2.1, 3.2.2, 3.3.1]{PanRongaVust}). In particular they show that $\mathfrak{Bir}_2(\mathbb{P}^3_\mathbb{C})$ has three irreducible components of dimension $26$, $28$, $29$; the component of dimension $26$ (resp. $28$, resp.~$29$) corresponds to birational maps of bidegree $(2,4)$ (resp. $(2,3)$, resp. $(2,2)$). We will see that the situation is slightly different for $\mathfrak{Bir}_3(\mathbb{P}^3_\mathbb{C})$; in particular we cannot expect such an explicit list of biclasses because there are infinitely many of biclasses (already the dimension of the family $\mathcal{E}_2$ of the classic cubo-cubic example is $39$ that is strictly larger that $\dim(\mathrm{PGL}(4;\mathbb{C})\times\mathrm{PGL}(4;\mathbb{C}))=30$). That's why the approach is different. 

\medskip

We do not have such a precise description of $\mathfrak{Bir}_d(\mathbb{P}^3_\mathbb{C})$ for $d\geq 4$. Nevertheless we can find a very fine and classical contribution for $\mathfrak{Bir}_3(\mathbb{P}^3_\mathbb{C})$ due to \textsc{Hudson} (\cite{Hudson}); in \S\ref{Sec:hudsontable} we reproduce Table~VI of \cite{Hudson}. \textsc{Hudson} introduces there some invariants to establish her classification. But it gives rise to many cases, and we also find examples where invariants take values that do not appear in her table. We do not know references explaining how her families fall into irreducible components of $\mathrm{Bir}_{3,d}(\mathbb{P}^3_\mathbb{C})$ so we focus on this natural question.

\begin{defi}
An element $\psi$ of $\mathrm{Bir}_{3,d}(\mathbb{P}^3_\mathbb{C})$ is \textbf{\textit{ruled}} if the strict transform of a generic plane under~$\psi^{-1}$ is a ruled cubic surface. 
\end{defi}

Denote by $\mathfrak{ruled}_{3,d}$ the set of $(3,d)$ ruled maps. Let us remark that there are no ruled birational maps of bidegree $(3,d)$ with $d\geq 6$. We detail $\mathfrak{ruled}_{3,d}$ in Lemma \ref{Lem:ruled}.

We describe the irreducible components of $\mathrm{Bir}_{3,d}(\mathbb{P}^3_\mathbb{C})$ for $3\leq d\leq 5$. Let us recall that the inverse of an element of $\mathrm{Bir}_{3,2}(\mathbb{P}^3_\mathbb{C})$ is quadratic and treated in \cite{PanRongaVust}.

\begin{theoalph}\label{thmA}
Assume that $2\leq d\leq 5$. The set $\mathfrak{ruled}_{3,d}$ is an irreducible component of $\mathrm{Bir}_{3,d}(\mathbb{P}^3_\mathbb{C})$. 

In bidegree $(3,3)$ $($resp. $(3,4)$$)$ there is only an other irreducible component; in bidegree $(3,5)$ there are three others.

The set $\mathfrak{ruled}_{3,3}$ intersects the closure of any irreducible component of $\overline{\mathrm{Bir}_{3,4}(\mathbb{P}^3_\mathbb{C})}$ $($the closures being taken in $\mathfrak{Bir}_3(\mathbb{P}^3_\mathbb{C}))$.
\end{theoalph}

\begin{nota}\label{nota:C2}
Consider a dominant rational map $\psi$ from $\mathbb{P}^3_\mathbb{C}$ into itself. For a generic line $\ell$, the preimage of $\ell$ by $\psi$ is a complete intersection $\Gamma_\ell$; let $\mathcal{C}_2$ be the union of the irreducible components of $\Gamma_\ell$ supported in the base locus of $\psi$. Define $\mathcal{C}_1$ by liaison from $\mathcal{C}_2$ in $\Gamma_\ell$. Remark that if $\psi$ is birational, then $\mathcal{C}_1=\psi^{-1}_*(\ell)$. Let us denote by $\mathfrak{p}_a(\mathcal{C}_i)$ the arithmetic genus of $\mathcal{C}_i$.
\end{nota}

It is difficult to find a uniform approach to classify elements of $\mathfrak{Bir}_3(\mathbb{P}^3_\mathbb{C})$. Nevertheless in small genus we succeed to obtain some common detailed results; before stating them, let us introduce some notations.

Let us remark that the inequality $\deg \psi^{-1}\leq(\deg\psi)^2$ mentioned previously directly follows from 
\[
(\deg \psi)^2=\deg\psi^{-1}+\deg \mathcal{C}_2.
\]

\begin{proalph}\label{propB}
Let $\psi$ be a $(3,d)$ birational map.

Assume that $\psi$ is not ruled, and $\mathfrak{p}_a(\mathcal{C}_1)=0$, {\it i.e. $\mathcal{C}_1$ is smooth}. Then
\begin{itemize}
\item[$\bullet$] $d\leq 6$; 

\item[$\bullet$] and $\mathcal{C}_2$ is a curve of degree $9-d$, and arithmetic genus $9-2d$.
\end{itemize}

Suppose $\mathfrak{p}_a(\mathcal{C}_1)=1$, and $2\leq d\leq 6$. Then
\begin{itemize}
\item[$\bullet$] there exists a singular point $p$ of $\mathcal{C}_1$ independent of the choice of $\mathcal{C}_1$;

\item[$\bullet$] if $d\leq 4$, all the cubic surfaces of the linear system $\Lambda_\psi$ are singular at $p$; 

\item[$\bullet$] the curve $\mathcal{C}_2$ is of degree $9-d$, of arithmetic genus $10-2d$, and lies on a unique quadric $Q$; more precisely $\mathcal{I}_{\mathcal{C}_2}=(Q,\mathcal{S}_1,\ldots,\mathcal{S}_{d-2})$ where the $\mathcal{S}_i$'s are independent cubics mo\-dulo $Q$.
\end{itemize}
\end{proalph}

We denote by $\mathrm{Bir}_{3,d,\mathfrak{p}_2}(\mathbb{P}^3_\mathbb{C})$ the subset of {\bf non-ruled} $(3,d)$ birational maps such that~$\mathcal{C}_2$ is of degree $9-d$, and arithmetic genus $\mathfrak{p}_2$. One has the following statement:

\begin{theoalph}\label{thmC}
If $\mathfrak{p}_2\in\{3,\,4\}$, then $\mathrm{Bir}_{3,3,\mathfrak{p}_2}(\mathbb{P}^3_\mathbb{C})$ is non-empty, and irreducible; $\mathrm{Bir}_{3,3,\mathfrak{p}_2}(\mathbb{P}^3_\mathbb{C})$ is empty as soon as $\mathfrak{p}_2\not\in\{3,\,4\}$.

If $\mathfrak{p}_2\in\{1,\,2\}$, then $\mathrm{Bir}_{3,4,\mathfrak{p}_2}(\mathbb{P}^3_\mathbb{C})$ is non-empty, and irreducible; $\mathrm{Bir}_{3,4,\mathfrak{p}_2}(\mathbb{P}^3_\mathbb{C})$ is empty as soon as $\mathfrak{p}_2\not\in\{1,\,2\}$.

The set $\mathrm{Bir}_{3,5,\mathfrak{p}_2}(\mathbb{P}^3_\mathbb{C})$ is empty as soon as $\mathfrak{p}_2\not\in\{-1,\,0,\,1\}$ and 
\begin{itemize}
\item[$\bullet$] if $\mathfrak{p}_2=-1$, then $\mathrm{Bir}_{3,5,\mathfrak{p}_2}(\mathbb{P}^3_\mathbb{C})$ is non-empty, and irreducible;

\item[$\bullet$] if $\mathfrak{p}_2=0$, then $\mathrm{Bir}_{3,5,\mathfrak{p}_2}(\mathbb{P}^3_\mathbb{C})$ is non-empty, and has two irreducible components;

\item[$\bullet$] if $\mathfrak{p}_2=1$, then $\mathrm{Bir}_{3,5,\mathfrak{p}_2}(\mathbb{P}^3_\mathbb{C})$ is non-empty, and has three irreducible components.
\end{itemize}
\end{theoalph}

\subsection*{Organization of the article} In \S \ref{Sec:defnot} we explain the particular case of ruled birational maps and set some notations. Then \S \ref{Sec:galfact} is devoted to liaison theory that plays a big role in the description of the irreducible components of $\mathrm{Bir}_{3,3}(\mathbb{P}^3_\mathbb{C})$ (\emph{see} \S\ref{Sec:cubocubique}), $\mathrm{Bir}_{3,4}(\mathbb{P}^3_\mathbb{C})$ (\emph{see} \S\ref{Sec:cuboquartic}) and $\mathrm{Bir}_{3,5}(\mathbb{P}^3_\mathbb{C})$ (\emph{see} \S\ref{Sec:cuboquintic}). In the last section we give some illustrations of invariants considered by \textsc{Hudson}, especially concerning the local study of the preimage of a line. Since \textsc{Hudson}'s book is very old, let us recall her classification in the first part of the appendix. 

\subsection*{Acknowledgment} The authors would like to thank J\'er\'emy \textsc{Blanc} and the referee for their helpful comments.

\section{Definitions, notations and first properties}\label{Sec:defnot}

\subsection{Definitions and notations}\label{subsec:def}

Let $\psi\colon\mathbb{P}^3_\mathbb{C}\dashrightarrow\mathbb{P}^3_\mathbb{C}$ be a rational map given, for some choice of coordinates, by  
\[
(z_0:z_1:z_2:z_3)\dashrightarrow\big(\psi_0(z_0,z_1,z_2,z_3):\psi_1(z_0,z_1,z_2,z_3):\psi_2(z_0,z_1,z_2,z_3):\psi_3(z_0,z_1,z_2,z_3)\big) 
\]
where the $\psi_i$'s are homogeneous polynomials of the same degree $d$, and without common factors. The map~$\psi$ is called a \textbf{\textit{\textsc{Cremona} transformation}} or a \textbf{\textit{birational map of $\mathbb{P}^3_\mathbb{C}$}} if it has a rational inverse~$\psi^{-1}$. The \textbf{\textit{degree}} of $\psi$, denoted $\deg\psi$, is $d$. The pair $(\deg\psi,\deg\psi^{-1})$ is the \textbf{\textit{bidegree}} of~$\psi$, we say that~$\psi$ is a $(\deg\psi,\deg\psi^{-1})$ birational map. The \textbf{\textit{indeterminacy set}} of $\psi$ is the set of the common zeros of the $\psi_i$'s. Denote by $\mathcal{I}_\psi$ the ideal generated by the $\psi_i$'s, and by $\Lambda_\psi\subset\mathrm{H}^0\big(\mathcal{O}_{\mathbb{P}^3_\mathbb{C}}(d)\big)$ the subspace of dimension $4$ generated by the $\psi_i$, and by $\deg\mathcal{I}_\psi$ the degree of the scheme defined by the ideal $\mathcal{I}_\psi$. The scheme whose ideal is $\mathcal{I}_\psi$ is denoted $F_\psi$. It is called \textbf{\textit{base locus}} of $\psi$. If~$\dim F_\psi=0$ than $F_\psi^1=\emptyset$, otherwise $F_\psi^1$ is the maximal subscheme of $F_\psi$ of dimension $1$ without isolated point, and without embedded point. Furthermore if $\mathcal{C}_i$ is a curve, then $\omega_{\mathcal{C}_i}$ is its dualizing sheaf. 

\begin{rem}
The second condition can also be stated as follows: $\mathbb{P}^3_\mathbb{C}\dashrightarrow \vert\mathcal{J}(3)\vert=\mathbb{P}^{3+k}_\mathbb{C}$ has an image of dimension $3$, and degree $k+1$.
\end{rem}

Let us give a few comments about Table~VI of \cite{Hudson}. For any subscheme $X$ of $\mathbb{P}^3_\mathbb{C}$ denote by $\mathcal{I}_X$ the ideal of $X$ in $\mathbb{P}^3_\mathbb{C}$. Let $\psi$ be a $(3,d)$ birational map. A point~$p$ is a \textbf{\textit{double point}} if all the cubic surfaces of $\Lambda_\psi$ are singular at~$p$. A point $p$ is a \textbf{\textit{binode}} if all the cubic surfaces of $\Lambda_\psi$ are singular at~$p$ with order $2$ approximation at~$p$ a quadratic form of rank $\leq 2$ (but this quadratic form is allowed to vary in $\Lambda_\psi$). In other words $p$ is a binode if there is a degree $1$ element $h$ of $\mathcal{I}_p$ such that all the cubics belong to $(h\cdot\mathcal{I}_p)+\mathcal{I}_p^3$. A point $p$ is a \textbf{\textit{double point of contact}} if the general element of $\Lambda_\psi$ is singular at~$p$ with order~$2$ approximation at $p$ a quadratic form generically constant on $\Lambda_\psi$. In other words $p$ is a double point of contact if all the cubics belong to $\mathcal{I}_p^3+(Q)$ with $Q$ of degree~$2$ and singular at $p$. A point~$p$ is a \textbf{\textit{point of contact}} if all the cubics belong to $\mathcal{I}_p^2+(\mathcal{S})$ where $\mathcal{S}$ is a cubic smooth at $p$. A point $p$ is a \textbf{\textit{point of osculation}} if all the cubics belong to $\mathcal{I}_p^3+(\mathcal{S})$ where $\mathcal{S}$ is a cubic smooth at $p$.

\begin{nota}
We will denote by $\mathcal{E}_i$ the $i$-th family of Table VI and by $\mathbb{C}[z_0,z_1,\ldots,z_n]_d$ the set of homogeneous polynomials of degree $d$ in the variables $z_0$, $z_1$, $\ldots$, $z_n$.
\end{nota}

\subsection{First properties}

Let us now focus on particular birational maps that cannot be dealt as the others: the ruled birational maps of $\mathbb{P}^3_\mathbb{C}$. Recall that there are two projective models of irreducible ruled cubic surfaces ; they both have the same normalization: $\mathbb{P}^2_\mathbb{C}$ blown up at one point which can be realized as a cubic surface in $\mathbb{P}^4_\mathbb{C}$ (\emph{see} \cite[Chapter $10$, introduction of \S\, 4.4]{Dolgachev:book}, \cite[Chapter $9$, \S\, 2.1]{Dolgachev:book}).

\begin{lem}\label{Lem:ruled}
Assume that $2\leq d\leq 5$.
\begin{itemize}
\item[$\bullet$] The set $\mathfrak{ruled}_{3,d}$ is irreducible.

\item[$\bullet$] Let $\psi$ be a general element of $\mathfrak{ruled}_{3,d}$, and let $\delta$ be the common line to all elements of $\big\{\mathrm{Sing}\,\mathcal{S}\,\vert\,\mathcal{S}\in\Lambda_\psi\big\}$; then
\[
\mathcal{I}_\psi=\mathcal{I}_\delta^2\cap\mathcal{I}_{\Delta_1}\cap\mathcal{I}_{\Delta_2}\cap\ldots\cap\mathcal{I}_{\Delta_{5-d}}\cap\mathcal{I}_K
\]
where $\Delta_i$ are disjoint lines that intersect $\delta$ at a unique point, and $K$ is a general reduced scheme of length $2d-4$.
\end{itemize}
\end{lem}

\begin{proof}
Let $\psi$ be an element of $\mathfrak{ruled}_{3,d}$. Recall that $F_\psi^1$ is the maximal subscheme of $F_\psi$ of dimension $1$ without isolated point, and without embedded point, \emph{i.e.} $F_\psi^1$ is a curve locally \textsc{Cohen}-\textsc{Macaulay}. Let us define $\mathcal{I}_K$ by: $\mathcal{I}_K=(\mathcal{I}_\psi:\mathcal{I}_{F_\psi^1})$.

An irreducible element $\mathcal{S}$ of $\Lambda_\psi$ is a ruled surface; it is also the projection of a ruled surface~$\widetilde{\mathcal{S}}$ of $\mathbb{P}^4_\mathbb{C}$. Recall that $\widetilde{\mathcal{S}}$ is also the blow-up $\widetilde{\mathbb{P}^2_\mathbb{C}}(p)$ of $\mathbb{P}^2_\mathbb{C}$ at $p$ embedded by $\vert\mathcal{I}_p(2h)\vert$, where $h$ is the class of an hyperplane in $\mathbb{P}^2_\mathbb{C}$. Let us denote by $\pi$ the projection $\widetilde{\mathcal{S}}\to\mathcal{S}$, by $H$ the class of an hyperplane of $\mathbb{P}^4_\mathbb{C}$, and by $E_p$ the exceptional divisor associated to the blow-up of $p$. Set $\widetilde{\delta}=\pi^{-1}\delta$, $\widetilde{\mathcal{C}_1}=\pi^{-1}(\mathcal{C}_1)$, and $\widetilde{F_\psi^1}=\pi^{-1}(F_\psi^1)$. One has 
\[
\widetilde{\delta}=h,\qquad H=2h-E_p, \qquad f=h-E_p,\qquad \widetilde{F_\psi^1}=2\widetilde{\delta}+D
\]
where $D$ is an effective divisor.
As $\widetilde{\mathcal{C}_1}+\widetilde{F_\psi^1}=3H$, $\widetilde{\mathcal{C}_1}\cdot f=1$ and $\widetilde{\mathcal{C}_1}\cdot H=d$ one gets $D\cdot f=0$, and $D\cdot H=5-d$; therefore $D=(5-d)f$. And we conclude that $\psi$ has a residual base scheme of length $2d-4$ from $\widetilde{\mathcal{C}_1}^2=2d-3$.

\bigskip

Conversely, take a general element of $\vert\mathcal{O}_{\widetilde{S}}\big((5-d)f\big)\vert$ and $2d-4$ general points on $\widetilde{S}$ of ideal~$\mathcal{I}$. We have $\mathrm{h}^0\big(\mathcal{I}(\widetilde{\mathcal{C}_1})\big)=3$, and thanks to the surjection $\mathrm{H}^0\mathcal{O}_{\mathbb{P}^3_\mathbb{C}}(3)\twoheadrightarrow \mathrm{H}^0\mathcal{O}_S(3)$ we get an element of~$\mathfrak{ruled}_{3,d}$.
\end{proof}

\begin{lem}\label{Lem:inclruled}
The following inclusions hold: 
\[
\mathfrak{ruled}_{3,2}\subset\overline{\mathfrak{ruled}_{3,3}},\qquad \mathfrak{ruled}_{3,3}\subset\overline{\mathfrak{ruled}_{3,4}},\qquad \mathfrak{ruled}_{3,4}\subset\overline{\mathfrak{ruled}_{3,5}}.
\]
\end{lem}

\begin{proof}[Proof $($with the notations introduced in the proof of Lemma \ref{Lem:ruled}$)$] 
Let us start with an element of $\overline{\mathfrak{ruled}_{3,5}}$ with base curve $\delta^2$ and $6$ base points $p_i$ in general position as decribed in Lemma \ref{Lem:ruled}. Then move two of the $p_i$, for instance $p_1$, $p_2$ until the line $(p_1p_2)$ intersects $\delta$. The line $(p_1p_2)$ is now automatically in the base locus of the linear system $\Lambda_\psi$, and we obtain like this a generic element of~$\mathfrak{ruled}_{3.4}$.

A similar argument allows to prove the two other inclusions. 
\end{proof}

Let us recall the notion of genus of a birational map (\cite[Chapter IX]{Hudson}). The \textbf{\textit{genus}} $\mathfrak{g}_\psi$ of~$\psi\in\mathrm{Bir}(\mathbb{P}^3_\mathbb{C})$ is the geometric genus of the curve $h\cap \psi^{-1}(h')$ where $h$ and $h'$ are generic hyperplanes of~$\mathbb{P}^3_\mathbb{C}$. The equality $\mathfrak{g}_\psi=\mathfrak{g}_{\psi^{-1}}$ holds.

\begin{rem} 
If $\psi$ is a birational map of $\mathbb{P}^3_\mathbb{C}$ of degree $1$ (resp. $2$, resp. $3$) then $\mathfrak{g}_\psi=0$ (resp.  $\mathfrak{g}_\psi=0$, resp. $\mathfrak{g}_\psi\leq 1$).
\end{rem}

One can give a characterization of ruled maps of $\mathrm{Bir}_{3,d}(\mathbb{P}^3_\mathbb{C})$ in terms of the genus.

\begin{pro}
Let $\psi$ be in $\mathrm{Bir}_{3,d}(\mathbb{P}^3_\mathbb{C})$, $2\leq d\leq 5$. The genus of $\psi$ is zero if and only if $\psi$ is ruled.
\end{pro}

\begin{proof}
On the one hand the base scheme of an element of~$\mathrm{Bir}_{3,d}(\mathbb{P}^3_\mathbb{C})$ has at most isolated singularities if and only if the map is not ruled; on the other hand $\mathfrak{g}_\psi=0$ if and only if for generic hyperplanes $h$, $h'$ of $\mathbb{P}^3_\mathbb{C}$ the curve $h\cap\psi^{-1}(h')$ is a singular rational cubic.
\end{proof}

\section{Liaison}\label{Sec:galfact}

According to \cite{PeskineSzpiro} we say that two curves $\Gamma_1$ and $\Gamma_2$ of $\mathbb{P}^3_\mathbb{C}$ are \textbf{\textit{geometrically linked}} if 
\begin{itemize}
\item[$\bullet$] $\Gamma_1\cup\Gamma_2$ is a complete intersection,

\item[$\bullet$] $\Gamma_1$ and $\Gamma_2$ have no common component.
\end{itemize}

\smallskip

Let $\Gamma_1$ and $\Gamma_2$ be two curves geometrically linked. Recall that $\mathcal{I}_{\Gamma_1\cup\Gamma_2}=\mathcal{I}_{\Gamma_1}\cap\mathcal{I}_{\Gamma_2}$. According to \cite[Proposition 1.1]{PeskineSzpiro} one has $\frac{\mathcal{I}_{\Gamma_1}}{\mathcal{I}_{\Gamma_1\cup\Gamma_2}}=\mathrm{Hom}\big(\mathcal{O}_{\Gamma_2},\mathcal{O}_{\Gamma_1\cup\Gamma_2}\big)$. Since the kernel of $\mathcal{O}_{\Gamma_1\cup\Gamma_2}\longrightarrow\mathcal{O}_{\Gamma_2}$ is $\frac{\mathcal{I}_{\Gamma_1}}{\mathcal{I}_{\Gamma_1\cup\Gamma_2}}$ one gets the following fundamental statement: if $\Gamma_1$, $\Gamma_2$ are two curves geometrically linked, then
\[
0\longrightarrow \omega_{\Gamma_1}\longrightarrow\omega_{\Gamma_1\cup \Gamma_2}\longrightarrow \mathcal{O}_{\Gamma_2}\otimes\omega_{\Gamma_1\cup \Gamma_2}\longrightarrow 0.
\]

\medskip

\begin{lem}\label{Lem:liaison}
Let $\psi$ be a rational map of $\mathbb{P}^3_\mathbb{C}$ of degree $3$. We have
\[
\omega_{\mathcal{C}_1\cup \mathcal{C}_2}=\mathcal{O}_{\mathcal{C}_1\cup \mathcal{C}_2}(2h),
\]
where $h$ denotes an hyperplane of $\mathbb{P}^3_\mathbb{C}$, and for $i\in\{1,\,2\}$
\begin{equation}\label{eq:blabla}
0\longrightarrow \omega_{\mathcal{C}_i}\longrightarrow\mathcal{O}_{\mathcal{C}_i\cup\mathcal{C}_{3-i}}(2h)\longrightarrow\mathcal{O}_{\mathcal{C}_{3-i}}(2h)\longrightarrow 0
\end{equation}
and
\begin{equation}\label{eq:blabla2}
0\longrightarrow \mathcal{I}_{\mathcal{C}_1\cup\mathcal{C}_2}(3h)\longrightarrow\mathcal{I}_{\mathcal{C}_i}(3h)\longrightarrow\omega_{\mathcal{C}_{3-i}}(h)\longrightarrow 0
\end{equation}
\end{lem}

\medskip

The first exact sequence $(\ref{eq:blabla})$ directly implies the following equalities ($i\in\{1,\,2\}$)
\[
\mathrm{H}^0\big(\omega_{\mathcal{C}_i}(-h)\big)=\mathrm{H}^0\big(\mathcal{I}_{\mathcal{C}_{3-i}}(h)\big),\qquad \mathrm{H}^0\big(\omega_{\mathcal{C}_i})=\mathrm{H}^0\big(\mathcal{I}_{\mathcal{C}_{3-i}}(2h)\big),
\]
\[
\mathrm{h}^0\omega_{\mathcal{C}_i}(h)+2=\mathrm{h}^0\big(\mathcal{I}_{\mathcal{C}_{3-i}}(3h)\big),\qquad\mathrm{H}^0\big(\omega_{\mathcal{C}_i}(h)\big)=\frac{\mathrm{H}^0\big(\mathcal{I}_{\mathcal{C}_{3-i}}(3h)\big)}{\mathrm{H}^0\big(\mathcal{I}_{\mathcal{C}_1\cup\mathcal{C}_2}(3h)\big)}.
\]

\medskip

\begin{cor}\label{cor:cubics}
Let $\psi$ be a rational map of $\mathbb{P}^3_\mathbb{C}$ of degree $3$. The ideal $\mathcal{I}_{\mathcal{C}_{3-i}}$ is generated by cubics if and only if $\omega_{\mathcal{C}_i}(h)$ is globally generated.
\end{cor}

\begin{proof}
It directly follows from the exact sequence $(\ref{eq:blabla2})$.
\end{proof}

\begin{cor}\label{cor:formule}
Let $\psi$ be a rational map of $\mathbb{P}^3_\mathbb{C}$ of degree $3$. Then 
\[
\deg \mathcal{C}_2-\deg\mathcal{C}_1=\mathfrak{p}_a(\mathcal{C}_2)-\mathfrak{p}_a(\mathcal{C}_1).
\]
\end{cor}

\begin{proof}
Taking the restriction of (\ref{eq:blabla}) to $\mathcal{C}_i$ for $i=1$, $2$ gives 
\[
\deg\omega_{\mathcal{C}_i}=2\deg\mathcal{C}_i-\deg(\mathcal{C}_1\cap\mathcal{C}_2),
\]
and hence
\[
\deg \mathcal{C}_2-\deg\mathcal{C}_1=\mathfrak{p}_a(\mathcal{C}_2)-\mathfrak{p}_a(\mathcal{C}_1).
\]
\end{proof}

Furthermore when $\mathcal{C}_1$ and $\mathcal{C}_2$ have no common component, and $\omega_{\mathcal{C}_i}$ is locally free, then $\mathrm{length}\,(\mathcal{C}_1\cap\mathcal{C}_2)=\deg\omega^{\vee}_{\mathcal{C}_i}(2h)$, {\it i.e.}
\[
\sum_{p\in\mathcal{C}_1\cap\mathcal{C}_2}\text{length}(\mathcal{C}_1\cap\mathcal{C}_2)_{\{p\}}=2\deg\mathcal{C}_i-2\mathfrak{p}_a(\mathcal{C}_i)+2.
\]

In the preimage of a generic point of $\mathbb{P}^3_\mathbb{C}$ by $\psi$, the number of points that do not lie in the base locus is given by
\[
3\deg \mathcal{C}_1-\sum_{p\in \mathcal{C}_1\cap \mathcal{C}_2}\mathrm{length}(\mathcal{S}\cap\mathcal{C}_1)_{\{p\}}-\sum_{p\in\Theta}\mathrm{length}(\mathcal{S}\cap\mathcal{C}_1)_{\{p\}}
\]
where $\mathcal{S}\in\Lambda_\psi$ is non-zero modulo $\mathrm{H}^0\big(\mathcal{I}_{\mathcal{C}_1\cup\mathcal{C}_2}(3h)\big)$, and where $\Theta$ denotes the set of irreducible components of dimension $0$ of the base locus $F_\psi$ of $\psi$.

\begin{lem}\label{lem:bir}
Let $\psi$ be a rational map of $\mathbb{P}^3_\mathbb{C}$ of degree $3$. Let $\Theta$ denote the set of irreducible components of dimension $0$ of $F_\psi$. The map $\psi$ is birational if and only if
\[
1=3\deg \mathcal{C}_1-\sum_{p\in \mathcal{C}_1\cap \mathcal{C}_2}\mathrm{length}(\mathcal{S}\cap\mathcal{C}_1)_{\{p\}}-\sum_{p\in\Theta}\mathrm{length}(\mathcal{S}\cap\mathcal{C}_1)_{\{p\}}.
\]
\end{lem}

Remark that the computation of $\mathrm{length}(\mathcal{S}\cap\mathcal{C}_1)_{\{p\}}$ depends on the nature of the singularity of the cubic surface and on the behavior of $\mathcal{C}_2$ in that point (\emph{see} \S\ref{Sec:singcubsurf}).

\begin{lem}\label{Lem:notonaplane}
Let $\psi$ be a $(3,d)$ \textsc{Cremona} map. Assume that $d\geq 4$, then $\mathcal{C}_1$ is not contained in a plane.
\end{lem}

\begin{proof}
Suppose for example that $d=4$; then $\mathcal{C}_1$ is contained in an irreducible cubic surface~$\mathcal{S}$. If~$\mathcal{C}_1$ is contained in a plane $\mathcal{P}$ then all the lines in $\mathcal{P}$ are quadrisecant to $\mathcal{S}$: contradiction with the irreducibility of $\mathcal{S}$.
\end{proof}

\begin{lem}\label{Lem:psurc2}
Let $\psi$ be a $(3,d)$ birational map, and let $p$ be a point on $\mathcal{C}_1$. Assume that the degree of the tangent cone of $\mathcal{C}_1$ at $p$ is strictly less than $4$. If any $\mathcal{S}$ in $\Lambda_\psi$ is singular at $p$, then $p$ belongs to $\mathcal{C}_2$.
\end{lem}

\begin{proof}
If any $\mathcal{S}$ in $\Lambda_\psi$ is singular at $p$, then the degree of the tangent cone of the complete intersection $\mathcal{C}_1\cap\mathcal{C}_2$ at $p$ is at least $4$ so $p$ has to belong to $\mathcal{C}_2$.
\end{proof}

\begin{lem}\label{Lem:indpt}
Let $\psi$ be a non-ruled $(3,d)$ birational map, and let $\mathcal{C}_1$ be a general element of $\Lambda_\psi$. The support of $\mathrm{Sing}\,\mathcal{C}_1$ is independent of the choice of $\mathcal{C}_1$.
\end{lem}

\begin{proof}
Let us show that there is a singular point independent of the choice of $\mathcal{C}_1$. Let us consider an element $\mathcal{S}$ of $\Lambda_\psi$ with finite singular locus. Let $\pi\colon\widetilde{\mathcal{S}}\to\mathcal{S}$ be a minimal desingularization of~$\mathcal{S}$, and let $\widetilde{\mathcal{C}_1}$ be the strict transform of $\mathcal{C}_1$. The elements of~$\Lambda_\psi$ give a linear system in $\vert\mathcal{O}_\mathcal{\widetilde{S}}(\widetilde{\mathcal{C}_1})\vert$ whose base locus denoted $\Omega$ is finite. According to \textsc{Bertini}'s theorem applied on $\widetilde{\mathcal{S}}$ one has the inclusion $\mathrm{Sing}\,\mathcal{C}_1\subset\pi(\Omega)\cup\mathrm{Sing}\,\mathcal{S}$. The first assertion thus follows from the fact that $\Omega\cup\mathrm{Sing}\,\mathcal{S}$ is finite.
\end{proof}

\begin{thm}\label{Thm:genre1}
Let $\psi$ be a $(3,d)$ birational map, $2\leq d\leq 6$, that is not ruled. Assume that $\mathfrak{p}_a(\mathcal{C}_1)=~1$. Then
\begin{itemize}
\item[$\bullet$] there exists a singular point $p$ of $\mathcal{C}_1$ independent of the choice of $\mathcal{C}_1$;

\item[$\bullet$] if $d\leq 4$, all the cubic surfaces of the linear system $\Lambda_\psi$ are singular at $p$;

\item[$\bullet$] the curve $\mathcal{C}_2$ is of degree $9-d$, of arithmetic genus $10-2d$, and lies on a unique quadric~$Q$; more precisely $\mathcal{I}_{\mathcal{C}_2}=(Q,\mathcal{S}_1,\ldots,\mathcal{S}_{d-2})$ where the $\mathcal{S}_i$'s are independent cubics modulo $Q$.
\end{itemize}
\end{thm}

\begin{rem}
As soon as $d=5$ the second assertion is not true. Indeed for $d=5$ we obtain two families: one for which all the elements of $\Lambda_\psi$ are singular, and another one for which it is not the case~(\S\ref{Sec:cuboquintic}).
\end{rem}

\begin{proof}
The first assertion directly follows from Lemma \ref{Lem:indpt}.\smallskip

 Since $\mathfrak{p}_a(\mathcal{C}_1)=1$, the curve $\mathcal{C}_2$ lies on a unique quadric $\mathcal{Q}$. The arithmetic genus of $\mathcal{C}_2$ is obtained from $\deg\mathcal{C}_2-\deg\mathcal{C}_1=\mathfrak{p}_a(\mathcal{C}_2)-\mathfrak{p}_a(\mathcal{C}_1)$ (Corollary \ref{cor:formule}). \smallskip

As $\mathfrak{p}_a(\mathcal{C}_1)=1$, $\omega_{\mathcal{C}_1}(h)$ has no base point, and $\mathcal{I}_{\mathcal{C}_2}$ is generated by cubics (Corollary \ref{cor:cubics}). The number of cubics containing $\mathcal{C}_2$ independent modulo the multiple of $Q$ is $d-2$: the liaison sequence (Lemma \ref{Lem:liaison}) becomes
\[
0\longrightarrow\mathcal{O}_{\mathcal{C}_1}(h)\longrightarrow\mathcal{O}_{\mathcal{C}_1\cup\mathcal{C}_2}(3h)\longrightarrow\mathcal{O}_{\mathcal{C}_2}(3h)\longrightarrow 0
\]
one gets that
\[
\mathrm{h}^0\mathcal{O}_{\mathcal{C}_2}(3h)=\mathrm{h}^0\mathcal{O}_{\mathcal{C}_1\cup \mathcal{C}_2}(3h)-\mathrm{h}^0\mathcal{O}_{\mathcal{C}_1}(h)=18-d.
\]
This implies that
\[
\mathrm{h}^0\mathcal{I}_{\mathcal{C}_2}(3h)=20-\mathrm{h}^0\mathcal{O}_{\mathcal{C}_2}(3h)=d+2.
\]
If we put away the four multiples of $Q$ one obtains $d+2-4=d-2$ cubics, and finally $\mathcal{I}_{\mathcal{C}_2}=(Q,\mathcal{S}_1,\ldots,\mathcal{S}_{d-2})$.
\end{proof}

Corollary \ref{cor:formule} and Theorem \ref{Thm:genre1} imply Proposition \ref{propB}.

\begin{pro}\label{Pro:compirr}
For $2\leq d\leq 5$ the set $\mathfrak{ruled}_{3,d}$ is an irreducible component of $\mathrm{Bir}_{3,d}(\mathbb{P}^3_\mathbb{C})$.
\end{pro}

\begin{proof}
Let us use the notations introduced in Lemma \ref{Lem:ruled}. Note that $F_\psi^1\subset\mathcal{C}_2$. If $\psi\in\mathrm{Bir}_{3,d}(\mathbb{P}^3_\mathbb{C})$ is not ruled then at a generic point $p\in F_\psi^1$ there exists an element of $\Lambda_\psi$ smooth at $p$. Hence $F_\psi^1$ is locally complete intersection at $p$ and $\deg F_\psi^1=\deg\mathcal{C}_ 2$. In particular $\deg\mathcal{I}_\psi=9-d$.\smallskip

Consider now an element $\psi$ in $\mathfrak{ruled}_{3,d}$. There is a line $\ell$ such that $\ell\subset\mathrm{Sing}\,\mathcal{S}$ for any $\mathcal{S}\in\Lambda_\psi$; the set~$F_\psi^1$ has an irreducible component whose ideal is $\mathcal{I}_\ell^2$ and $F_\psi^1$ is not locally complete intersection. This multiple structure has to be contained in $\mathcal{C}_2$ but since $\mathcal{C}_2$ is locally complete intersection the inequality $\deg\mathcal{C}_2>\deg F_\psi^1$ holds; it can be rewritten $\deg \mathcal{I}_\psi<9-d$.\smallskip

The number $\deg\mathcal{I}_\psi$ cannot decrease by specialization so we cannot specialize a non-ruled birational map into a ruled one while staying in the same bidegree\footnote{As we will see in Proposition \ref{Pro:inter} this statement is not true if we do not specify ''while staying in the same degree''.}.\smallskip

Elements of $\Lambda_\psi$ when $\psi$ is a ruled birational map have no isolated singularities whereas general elements of $\Lambda_\psi$ when $\psi$ is a non-ruled birational map have at most isolated singularities, it is impossible to specialize a ruled birational map into a non-ruled one.
\end{proof}

\begin{cor}
Let $\psi$ be a $(3,\cdot)$ birational map of $\mathbb{P}^3_\mathbb{C}$; if the general element of $\Lambda_\psi$ is smooth or if the singularities of a general element of $\Lambda_\psi$ are isolated, then $\deg F_\psi^1=\deg\mathcal{C}_2$.
\end{cor}

\section{$(3,3)$-\textsc{Cremona} transformations}\label{Sec:cubocubique}

\subsection{Some known results}

\subsubsection{}

In the literature one can find different points of view concerning the classification of $(3,3)$ birational maps. For example \textsc{Hudson} introduced many invariants related to singularities of fa\-milies of surfaces and gave four families described in \S\ref{Sec:hudsontable}; nevertheless we do not understand why the family $\mathcal{E}_{3.5}$ defined below does not appear. \textsc{Pan} chose an other point of view and regrouped $(3,3)$ birational maps into three families. A $(3,3)$ birational map $\psi$ of $\mathbb{P}^3_\mathbb{C}$ is called \textbf{\textit{determinantal}} if there exists a $4\times 3$ matrix~$M$ with linear entries such that $\psi$ is given by the four $3\times 3$ minors of the matrix $M$; the inverse~$\psi^{-1}$ is also determinantal. Let us denote by $\mathbf{T}_{3,3}^{\mathbf{D}}$ the set of determinantal maps. A $(3,3)$ \textsc{Cremona} transformation is a \textbf{\textit{\textsc{de Jonqui\`eres}}} one if and only if the strict transform of a general line under $\psi^{-1}$ is a singular plane rational cubic curve whose singular point is fixed. For such a map  there is always a quadric contracted onto a point, the corresponding fixed point for~$\psi^{-1}$ which is also a \textsc{de Jonqui\`eres} transformation. The \textsc{de Jonqui\`eres} transformations form the set $\mathbf{T}_{3,3}^{\mathbf{J}}$. \textsc{Pan} established the following (\cite[Theorem 1.2]{Pan}): 
\[
\mathrm{Bir}_{3,3}(\mathbb{P}^3_\mathbb{C})=\mathbf{T}_{3,3}^{\mathbf{D}}\cup\mathbf{T}_{3,3}^{\mathbf{J}}\cup\mathfrak{ruled}_{3,3};
\]
in other words an element of $\mathrm{Bir}_{3,3}(\mathbb{P}^3_\mathbb{C})$ is a determinantal map, or a \textsc{de Jonqui\`eres} map, or a ruled map.

\begin{rem}
One has $\mathbf{T}_{3,3}^{\mathbf{D}}=\mathrm{Bir}_{3,3,3}(\mathbb{P}^3_\mathbb{C})$ and $\mathbf{T}_{3,3}^{\mathbf{J}}=\mathrm{Bir}_{3,3,4}(\mathbb{P}^3_\mathbb{C})$; hence $\mathrm{Bir}_{3,3,\mathfrak{p}_2}(\mathbb{P}^3_\mathbb{C})$ is irreducible for $\mathfrak{p}_2\in\{3,4\}$ (\emph{see} \cite{Pan}).
\end{rem}

\begin{rem}
The birational involution $(z_0z_1^2:z_0^2z_1:z_0^2z_2:z_1^2z_3)$ is determinantal, the matrix being 
\[
\left[
\begin{array}{ccc}
z_0 & z_3 & 0\\
-z_1 & 0 & z_2\\
0 & 0 & -z_1\\
0 & -z_0 & 0
\end{array}
\right],
\]
and also ruled: all the partial derivatives of the components of the map vanish on $z_0=z_1=0$. The \textsc{Cremona} transformation $(z_0^3:z_0^2z_1:z_0^2z_2:z_1^2z_3)$ is a \textsc{de Jonqui\`eres} and a ruled one.

One has (\cite{Pan2})
\[
\mathbf{T}_{3,3}^{\mathbf{D}}\cap\mathbf{T}_{3,3}^{\mathbf{J}}=\emptyset,\quad\quad
\mathbf{T}_{3,3}^{\mathbf{D}}\cap\mathfrak{ruled}_{3,3}\not=\emptyset,\quad\quad
\mathbf{T}_{3,3}^{\mathbf{J}}\cap\mathfrak{ruled}_{3,3}\not=\emptyset.
\]
\end{rem}

We deal with the natural description of the irreducible components of $\mathrm{Bir}_{3,3}$ which does not coincide with \textsc{Pan}'s point of view since one of his family is contained in the closure of another one.

\subsection{Irreducible components of the set of $(3,3)$ birational maps}

\subsubsection{General description of $(3,3)$ birational maps}

One already describes an irreducible component of $\mathrm{Bir}_{3,3}(\mathbb{P}^3_\mathbb{C})$, the one that contains $(3,3)$ ruled birational maps (Proposition \ref{Pro:compirr}). Hence let us consider the case where the linear system $\Lambda_\psi$ associated to $\psi\in\mathrm{Bir}_{3,3}(\mathbb{P}^3_\mathbb{C})$ contains a cubic surface without double line.

\begin{itemize}
\item If $\mathcal{C}_1$ is smooth then it is a twisted cubic, we are in family $\mathcal{E}_2$ of Table VI (\emph{see} \S\ref{Sec:hudsontable}). In that case~$\psi$ is determinantal; more precisely a $(3,3)$ birational map is determinantal if and only if its base locus scheme is an arithme\-tically \textsc{Cohen}-\textsc{Macaulay} curve of degree~$6$ and (arithmetic) genus $3$ (\emph{see} \cite[Proposition~1]{AvritzerGonzalezSprinbergPan}).

\medskip

\item Otherwise $\omega_{\mathcal{C}_1}=\mathcal{O}_{\mathcal{C}_1}$, and $\psi$ belongs to the irreducible family $\mathbf{T}_{3,3}^{\mathbf{J}}$ of \textsc{Jonqui\`eres} maps ($\mathcal{E}_3$ in terms of \textsc{Hudson}'s classification). The curve $\mathcal{C}_2$ lies on a quadric described by the quadratic form $Q$. According to Theorem \ref{Thm:genre1} the ideal of $\mathcal{C}_2$ is $(Q,\mathcal{S})$, and there exists a point~$p$ such that $p\in Q$, and $p$ is a singular point of $\mathcal{S}$. Furthermore $\mathcal{I}_\psi=\mathcal{I}_pQ+(\mathcal{S})$. Reciprocally such a triplet $(p,Q,\mathcal{S})$ induces a birational map. 

The family $\mathbf{T}_{3,3}^{\mathbf{J}}$ is stratified as follows by \textsc{Hudson} (all the cases belong to $\overline{\mathcal{E}_3}$):
\begin{itemize}
\item[$\bullet$] {\sl Description of $\mathcal{E}_3$}. The general element of $\mathcal{I}_pQ+(\mathcal{S})$ has an ordinary quadratic singula\-rity at $p$ (configuration $(2,2)$ of \texttt{Table} $1$ (\emph{see} \S\ref{Sec:singcubsurf})), and the generic cubic is singular at~$p$ with a quadratic form of rank $3$.

\item[$\bullet$]  {\sl Description of $\mathcal{E}_{3.5}$}. The point $p$ lies on $Q$ ($p$ is a smooth point or not) and the generic cubic is singular at $p$ with a quadratic form of rank $2$. In other words $p$ is a binode and this happens when one of the two biplanes is contained in $\mathrm{T}_pQ$, it corresponds to the configuration $(2,3)'$ of \texttt{Table} $1$ (\emph{see} \S\ref{Sec:singcubsurf}). 
The generic cubic is singular at $p$ with a quadratic form of rank $2$; this case does not appear in Table VI (\emph{see} \S\ref{Sec:hudsontable}). Let us denote by~$\mathcal{E}_{3.5}$ the set of the associated $(3,3)$ birational maps. The curve $\mathcal{C}_2$ has degree~$6$ and a triple point (in $Q$).

\item[$\bullet$] {\sl Description of $\mathcal{E}_4$}. 
The point $p$ is a double point of contact, it corresponds to configuration $(2,4)$ of \texttt{Table}~$1$ (\emph{see} \S\ref{Sec:singcubsurf}). 
\end{itemize}
\end{itemize}

\begin{pro}\label{Pro:incidence}
One has 
 \[
 \dim \mathcal{E}_2=39, \quad\dim\mathcal{E}_3=38, \quad\dim\mathcal{E}_{3.5}=35, \quad\dim\mathcal{E}_4=35, \quad \dim\mathcal{E}_5=31,
\]
and
 \[
\overline{\mathcal{E}_3}=\mathbf{T}_{3,3}^{\mathbf{J}}, \quad\mathring{\mathcal{E}_{3.5}}\subset\overline{\mathcal{E}_3},\quad\mathring{\mathcal{E}_4}\subset\overline{\mathcal{E}_3},\quad\mathring{\mathcal{E}_4}\not\subset\overline{\mathcal{E}_{3.5}},\quad\mathring{\mathcal{E}_{3.5}}\not\subset\overline{\mathcal{E}_4}.
\]
\end{pro}

\begin{proof}
Let us justify the equality $\dim\mathcal{E}_3=38$. We have to choose a quadric $Q$ and a point $p$ on~$Q$, this gives $9+2=11$. Then we take a cubic surface singular at $p$ that yields to $19-4=15$; since we look at this surface modulo $pQ$ one gets $15-3=12$ so 
\[
\dim\mathcal{E}_3=11+12+15=38.
\]

Let us deal with $\dim\mathcal{E}_4$. We take a singular quadric $Q$ this gives $8$. Then we take a cubic singular at $p$, modulo $pQ$ and this yields to $19-4-3=12$, and finally one obtains $12+8+15=35$.
\end{proof}

\subsubsection{Irreducible components}

\begin{thm}\label{thm:comp33}
The set $\mathfrak{ruled}_{3,3}$ is an irreducible component of $\mathrm{Bir}_{3,3}(\mathbb{P}^3_\mathbb{C})$, and there is only one another irreducible component in $\mathrm{Bir}_{3,3}(\mathbb{P}^3_\mathbb{C})$. More precisely the set of the Jonqui\`eres maps $\overline{\mathcal{E}_3}$ is contained in the closure of determinantal ones $\overline{\mathcal{E}_2}$ whereas $\mathfrak{ruled}_{3,3}\not\subset\overline{\mathcal{E}_2}$.
\end{thm}

\begin{proof}
Let us consider the matrix $A$ given by 
\[
\left[
\begin{array}{ccc}
0 & 0 & 0\\
-z_1 & -z_2 & 0\\
z_0 & 0 & -z_2\\
0 & z_0 & z_1
\end{array}
\right]
\]
and let $A_i$ denote the matrix $A$ minus the $(i+1)$-th line.
If $i>0$, the $2\times 2$ minors of $A_i$ are divisible by $z_{i-1}$. 

Consider the $3\times 4$ matrix $B$ given by $\big[b_{ij}\big]_{1\leq i\leq 4,\,1\leq j\leq 3}$
with $b_{ij}\in\mathrm{H}^0\big(\mathcal{O}_{\mathbb{P}^3_\mathbb{C}}(1)\big)$; as previously,~$B_i$ is the matrix $B$ minus the $(i+1)$-th line. Denote by $\Delta^{j,k}$ the determinant of the matrix $A_0$ minus the $j$-th line and the $k$-th column. The $\Delta^{j,k}$ generate $\mathbb{C}[z_0,z_1,z_2]_2$.
One has 
\[
\det(A_0+tB_0)=t\cdot S\quad [t^2]
\]
where 
\[
S=(b_{21}+b_{43})\Delta^{1,1}-(b_{31}-b_{42})\Delta^{2,1}+(b_{33}-b_{22})\Delta^{1,2}+b_{23}\Delta^{1,3}+b_{32}\Delta^{2,2}+b_{41}\Delta^{3,1} 
\]
is a generic cubic of the ideal $(z_0,z_1,z_2)^2$. For $i>0$ 
\[
\det(A_i+tB_i)=\det A_i+t\cdot(z_{i+1}Q)(-1)^{i+1}=t\cdot(z_{i+1}Q)(-1)^{i+1}\quad [t^2]
\]
where $Q=b_{1,1}z_2-b_{1,2}z_1+b_{1,3}z_0$ is the equation of a generic quadric that contains $(0,0,0,1)$.
So the map 
\[
\left[\frac{\det(A_0+tB_0)}{t}:\frac{\det(A_1+tB_1)}{t}:\frac{\det(A_2+tB_2)}{t}:\frac{\det(A_3+tB_3)}{t}\right] 
\]
allows to go from $\mathcal{E}_2$ to a general element of $\overline{\mathcal{E}_3}$.

\bigskip

Furthermore $\overline{\mathcal{E}_3}$ and $\mathfrak{ruled}_{3,3}$ are different components (Proposition \ref{Pro:compirr}).
\end{proof}

\section{$(3,4)$-\textsc{Cremona} transformations}\label{Sec:cuboquartic}

\subsection{General description of $(3,4)$ birational maps}\label{Subsec:cuboquartic}

The ruled maps $\mathfrak{ruled}_{3,4}$ give rise to an irreducible component (Proposition \ref{Pro:compirr}). Let us now focus on the case where the linear system $\Lambda_\psi$ associated to $\psi\in\mathrm{Bir}_{3,4}(\mathbb{P}^3_\mathbb{C})$ contains a cubic surface without double line.

\medskip

\begin{itemize}
\item First case: $\mathcal{C}_1$ is smooth. From $\mathrm{h}^0\omega_{\mathcal{C}_1}(h)=3$ one gets that $\mathcal{C}_2$ lies on five cubics. Since $\mathrm{h}^0\mathcal{O}_{\mathcal{C}_2}(2h)=~0$ the curve $\mathcal{C}_1$ lies on a quadric, and $\mathrm{h}^0\omega_{\mathcal{C}_2}=1$ thus $\omega_{\mathcal{C}_2}=\mathcal{O}_{\mathcal{C}_2}$. This configuration corresponds to $\mathcal{E}_6$. 

\medskip

\item Second case: $\mathcal{C}_1$ is a singular curve of degree $4$ not contained in a plane (\emph{see} Lemma~\ref{Lem:notonaplane}) so $\omega_{\mathcal{C}_1}=\mathcal{O}_{\mathcal{C}_1}$. The curve $\mathcal{C}_1$ lies on two quadrics and $\mathcal{C}_2$ on six cubics ($\mathrm{h}^0\omega_{\mathcal{C}_1}(h)=4$). Let~$p$ be the singular point of $\mathcal{C}_1$; all elements of $\Lambda_\psi$ are singular at $p$ (Theorem~\ref{Thm:genre1}), and~$p$ belongs to~$\mathcal{C}_2$ (Lemma \ref{Lem:psurc2}). The curve $\mathcal{C}_2$ lies on a unique quadric $\mathcal{Q}$ (Theorem \ref{Thm:genre1}), is linked to a line~$\ell$ in a $(2,3)$ complete intersection $Q\cap \mathcal{S}_1$ (with $\deg Q=2$ and $\deg \mathcal{S}_1=3$), and $\mathcal{I}_{\mathcal{C}_2}=(Q,\mathcal{S}_1,\mathcal{S}_2)$ with $\deg\mathcal{S}_2=3$ (Theorem \ref{Thm:genre1}). 

Since $\mathcal{C}_1$ is of degree $4$ and arithmetic genus $1$, one has $\mathrm{H}^0\big(\mathcal{O}_{\mathcal{C}_1}(h)\big)=\mathrm{H}^0\big(\mathcal{O}_{\mathbb{P}^3_\mathbb{C}}(1)\big)$. Let us consider $L=\mathrm{H}^0\big(\mathcal{I}_{\mathcal{C}_1\cup\mathcal{C}_2}(3h)\big)\subset\Lambda_\psi$ and the map
\[
\mathrm{H}^0\big(\mathcal{O}_{\mathcal{C}_1}(h)\big)\longrightarrow\frac{\mathrm{H}^0\big(\mathcal{I}_{\mathcal{C}_2}(3h)\big)}{L},\qquad h\mapsto Qh;
\]
it is injective. Indeed $\dim(\mathcal{C}_1\cap Q)=0$ thus modulo $Q$ the cubics defining $\mathcal{C}_1$ are independent. Therefore $\Lambda_\psi$ is contained in $(Q\mathcal{I}_p,\mathcal{S}_1,\mathcal{S}_2)$. For $p'\in\mathbb{P}^3_\mathbb{C}\smallsetminus Q$ one has
\[
\mathrm{H}^0\big((Q\mathcal{I}_p,\mathcal{S}_1,\mathcal{S}_2)\cap\mathcal{I}_{p'})(3)\big)=\Lambda_\phi 
\]
for some birational map $\phi$, and $\psi$ belongs to the closure of the set defined by all such maps~$\phi$.

\smallskip

Reciprocally let $Q$ be a quadric, $p$ be a point on $Q$, $\mathcal{S}_1$ be a cubic singular at $p$ and that contains a line $\ell$ of $Q$. If $\mathcal{C}_2$ is the residual of $\ell$ in $(Q,\mathcal{S}_1)$, then there exists $\mathcal{S}_2$ singular at $p$ such that $\mathcal{I}_{\mathcal{C}_2}=(Q,\mathcal{S}_1,\mathcal{S}_2)$. Set 
\[
\Lambda=\mathrm{H}^0\big((\mathcal{I}_{p_1}\cap(Q\mathcal{I}_p,\mathcal{S}_1,\mathcal{S}_2)(3)\big).
\]
Let $L$ be a $2$-dimensional general element of $\Lambda$; the general linked curve to $\mathcal{C}_2$ in $L$, denoted~$\mathcal{C}_{1,L}$, is of degree $4$, is singular at $p$, lies on two quadrics; furthermore the linear system induced by $\Lambda$ on $\mathcal{C}_{1,L}$ has the two following properties:
\begin{itemize}
\item its base locus contains $p$ and $p_1$,

\item it is birational.
\end{itemize}
In other words, $\Lambda=\Lambda_\psi$ for a $(3,4)$-birational map $\psi$.

\bigskip

Let us give some explicit examples, the generic one and the degeneracies considered by \textsc{Hudson}:
\begin{itemize}
\item[$\bullet$] {\sl Description of $\mathcal{E}_7$.} The quadric $Q$ is smooth at $p$, and the rank of $Q$ is maximal. Hence the point $p$ is an ordinary quadratic singularity of the generic element of $\Lambda_\psi$, we are in the configuration $(2,2)$ of \texttt{Table} $1$ (\emph{see} \S\ref{Sec:singcubsurf}). 

\item[$\bullet$] {\sl Description of $\mathcal{E}_{7.5}$.} In that case, $p$ is a binode, $Q$ is smooth at $p$ and one of the two biplanes is contained in $T_pQ$; we are in the configuration $(2,3)'$ of \texttt{Table} $1$ (\emph{see}~\S\ref{Sec:singcubsurf}). The set of such maps is denoted~$\mathcal{E}_{7.5}$, this case does not appear in Table~VI but should appear.

\item[$\bullet$] {\sl Description of $\mathcal{E}_8$.}  The second way to obtain a binode is the following one: $Q$ is an irreducible cone with vertex $p$. This corresponds to the configuration $(2,3)$ of \texttt{Table}~$1$ (\emph{see} \S\ref{Sec:singcubsurf}).

\item[$\bullet$] {\sl Description of $\mathcal{E}_9$.} The rank of $Q$ is $2$, and the point $p$ is a double point of contact; we are in the configuration $(2,4)$ of \texttt{Table} $1$ (\emph{see} \S\ref{Sec:singcubsurf}).

\item[$\bullet$] {\sl Description of $\mathcal{E}_{10}$.} The general element of $\Lambda_\psi$ has a double point of contact and a binode (configurations $(2,4)$ and $(1,4)$ of \texttt{Table} $1$, \emph{see} \S\ref{Sec:singcubsurf}). \textsc{Hudson} details this case carefully (\cite[Chap. XV]{Hudson}).
\end{itemize}
\end{itemize}

\begin{pro}
One has the following properties:
\[
\dim\mathcal{E}_6=38,\qquad\mathcal{E}_{7.5}\cup\mathcal{E}_8\subset\overline{\mathcal{E}_7}
\]
and 
\begin{itemize}
\item[$\bullet$] a generic element of $\mathcal{E}_{7.5}$ is not a specialization of a generic element of $\mathcal{E}_8$;

\item[$\bullet$] a generic element of $\mathcal{E}_8$ is not a specialization of a generic element of $\mathcal{E}_{7.5}$;

\item[$\bullet$] a generic element of $\mathcal{E}_9$ is a specialization of a generic element of $\mathcal{E}_8$.
\end{itemize}
\end{pro}

\begin{proof}
The arguments to establish $\dim\mathcal{E}_6=38$ are similar to those used in the proof of Proposition~\ref{Pro:incidence}.

Let us justify that a generic element of $\mathcal{E}_{7.5}$ is not a specialization of a generic element of $\mathcal{E}_8$ (we take the notations of \S\ref{Subsec:cuboquartic}): as we see when $\psi\in\mathcal{E}_8$ the quadric $Q$ is always singular whereas it is not the case when $\psi\in\mathcal{E}_{7.5}$. Conversely if $\psi$ belongs to $\mathcal{E}_{7.5}$ then $\mathcal{C}_2$ is reducible but if $\psi$ belongs to $\mathcal{E}_8$ the curve $\mathcal{C}_2$ can be irreducible and reduced; hence a generic element of $\mathcal{E}_8$ is not a specialization of a generic element of $\mathcal{E}_{7.5}$.
\end{proof}

\begin{thm}\label{thm:comp34}
The set $\mathfrak{ruled}_{3,4}$ is an irreducible component of $\mathrm{Bir}_{3,4}(\mathbb{P}^3_\mathbb{C})$. There is only one another irreducible component in $\mathrm{Bir}_{3,4}(\mathbb{P}^3_\mathbb{C})$.
\end{thm}

\begin{proof}
According to Proposition \ref{Pro:compirr} the set $\mathfrak{ruled}_{3,4}$ is an irreducible component of $\mathrm{Bir}_{3,4}(\mathbb{P}^3_\mathbb{C})$.

Any element $\psi$ of $\mathcal{E}_7\cup\mathcal{E}_{7.5}\cup\mathcal{E}_8\cup\mathcal{E}_9\cup\mathcal{E}_{10}$ satisfies the following property: 
\[
\Lambda_\psi=\mathrm{H}^0\big(((Q\mathcal{I}_p,\mathcal{S}_1,\mathcal{S}_2)\cap \mathcal{I}_{p_1})(3)\big)
\]
where $p$ belongs to $Q$, $p_1$ is an ordinary base point, and
\[
Q=\det\left[\begin{array}{cc}
L_0 & L_1\\
L_2 & L_3
\end{array}
\right],\quad \mathcal{S}_1=L_0Q_1+L_1Q_2,\quad \mathcal{S}_2=L_2Q_1+L_3Q_2 
\]
with $L_i\in\mathbb{C}[z_0,z_1,z_2,z_3]_1$, $Q_i\in\mathbb{C}[z_0,z_1,z_2]_2$. So $\mathcal{E}_7$, $\mathcal{E}_{7.5}$, $\mathcal{E}_8$ $\mathcal{E}_9$ and $\mathcal{E}_{10}$ belong to the same irreducible component $\mathscr{E}$. 

\medskip

It remains to show that $\mathscr{E}=\overline{\mathcal{E}_6}$: let us consider 
\[
J=\left[\begin{array}{cccc} 
0 & 0 & 0 & 1\\
0 & 0 & 1 & 0\\
0 & -1 & 0 & 0 \\
-1 & 0 & 0 & 0
\end{array}\right], \quad 
N=\left[\begin{array}{cccc}
0 & -z_2 & z_3 & L_0\\
z_2 & 0 & L_1 & L_2\\
-z_3 & -L_1 & 0 & L_3\\
-L_0 & -L_2 & -L_3 & 0
\end{array}\right],\quad 
v=\left[\begin{array}{c}
z_2\\
z_1\\
z_0\\
tz_3 \end{array}\right]
\]
with $L_i$ linear forms and
\[
M_t=\left[\begin{array}{c}
Jv\\
Nv
 \end{array}\right]=
 \left[\begin{array}{cccc}
 tz_3 & z_0 & -z_1 & -z_2\\
 tz_3L_0+Q & q_1 & q_2 & q_3
 \end{array}\right]
\]
with 
\begin{align*}
&Q=z_0z_3-z_1z_2,&& q_1=z_2^2+z_0L_1+tz_3L_2,\\
&q_2=-z_2z_3-z_1L_1+tz_3L_3,&& q_3=-z_2L_0-z_1L_2-z_0L_3.
\end{align*}
For generic $L_i$'s and $t\not=0$ the $2\times 2$ minors of $M_t$ generate the ideal of a generic elliptic quintic curve as in $\mathcal{E}_6$. For $M_0$ the $2\times 2$ minors become $Qz_0$, $Qz_1$, $Qz_2$, $\mathcal{S}_1$, $\mathcal{S}_2$, and $\mathcal{S}_3$ with 
\[
\mathcal{S}_1=-z_2Q,\qquad\mathcal{S}_2=-z_1q_3+z_2q_2,\qquad\mathcal{S}_3=z_0q_3+z_2q_1. 
\]
Therefore the ideal $\mathcal{M}_2$ generated by these minors is
\[
(Qz_0,Qz_1,Qz_2,\mathcal{S}_2,\mathcal{S}_3).
\]
Denote by $\ell$ the line defined by $\mathcal{I}_\ell=(z_1,z_3)$. According to
\[
z_3\mathcal{S}_3=-z_2\mathcal{S}_2+Q(q_3+L_1z_2)\quad\&\quad z_1\mathcal{S}_3=-z_0\mathcal{S}_2-z_2^2Q
\]
$\mathcal{M}_2$ is the ideal of the residual of $\ell$ in the complete intersection of ideals $(Q,\mathcal{S}_2)$. 

It only remains to prove that one can obtain the generic element of $\overline{\mathcal{E}_7}$ with a good choice of the $L_i$'s; in other words it remains to prove that $\mathcal{S}_2$ is generic among the cubics singular at $p$ that contain $\ell$. Modulo $Q$ one can assume that $q_3=-z_3a+b$, with $a$ (resp. $b$) an element of $\mathbb{C}[z_1,z_2]_1$ (resp. $\mathbb{C}[z_0,z_1,z_2]_2$). Then
\[
\mathcal{S}_2=-z_3\big(z_1a+z_2^2\big)+z_1\big(b-z_2L_1\big);
\]
in conclusion $\mathcal{S}_2=z_3A+z_2B$ for generic $A$ and $B$ in $\mathbb{C}[z_0,z_1,z_2]_2$.
\end{proof}

\subsection{Relations between $\overline{\mathrm{Bir}_{3,3}(\mathbb{P}^3_\mathbb{C})}$ and $\overline{\mathrm{Bir}_{3,4}(\mathbb{P}^3_\mathbb{C})}$}

One can now state the following result:

\begin{pro}\label{Pro:inter}
The set $\mathfrak{ruled}_{3,3}$ intersects the closure of any irreducible component of\, $\overline{\mathrm{Bir}_{3,4}(\mathbb{P}^3_\mathbb{C})}$.
\end{pro}

\begin{proof}
According to Lemma \ref{Lem:inclruled} it is sufficient to prove that $\mathfrak{ruled}_{3,3}$ intersects the closure of $(3,4)$ birational maps that are non-ruled.

Let us consider an element $\psi$ of $\mathrm{Bir}_{3,4}(\mathbb{P}^3_\mathbb{C})$ whose $\mathcal{C}_2$ is the union of the lines of ideals 
\[
\mathcal{I}_\delta=(z_0,z_1^2), \quad(z_0-\varepsilon z_2,z_1), \quad \mathcal{I}_{\ell_1}=(z_0,z_3), \quad \mathcal{I}_{\ell_2}=(z_1,z_2).
\] 
Denote by $\mathcal{J}_\varepsilon=(z_0,z_1^2)\cap(z_0-\varepsilon z_2,z_1)\cap(z_0,z_3)\cap(z_1,z_2)$. One can check that 
\[
\mathcal{J}_\varepsilon=(z_0z_1,z_0^2z_2+\varepsilon z_0z_2^2,z_1^2z_3).
\] 
Set $\mathcal{I}_\varepsilon=z_0z_1(z_0,z_1,z_2)+(z_0^2z_2+\varepsilon z_0z_2^2,z_1^2z_3)$. For a general $p_2$ the map $\psi_\varepsilon$ defined by $\Lambda_{\psi_\varepsilon}=\mathrm{H}^0\big((\mathcal{I}_\varepsilon\cap \mathcal{I}_{p_2})(3)\big)$ is birational; furthermore
\begin{itemize}
\item[$\bullet$] $\psi_\varepsilon\in\mathrm{Bir}_{3,4}(\mathbb{P}^3_\mathbb{C})\smallsetminus\mathfrak{ruled}_{3,4}$ for $\varepsilon\not=0$;

\item[$\bullet$] $\psi_0\in\mathfrak{ruled}_{3,3}$.
\end{itemize}
\end{proof}

As in the case of $(3,3)$ birational maps one has the following statement:

\begin{thm}
If $\mathfrak{p}_2\in\{1,\,2\}$, then $\mathrm{Bir}_{3,4,\mathfrak{p}_2}(\mathbb{P}^3_\mathbb{C})$ is non-empty and irreducible.
\end{thm}

\section{$(3,5)$-\textsc{Cremona} transformations}\label{Sec:cuboquintic}

\subsection{General description of $(3,5)$ birational maps}

We already find an irreducible component of the set of $(3,5)$ birational maps: $\mathfrak{ruled}_{3,5}$ (Proposition \ref{Pro:compirr}). Let us now consider a $(3,5)$-\textsc{Cremona} transformation $\psi$ such that $\Lambda_\psi$ contains a cubic surface without double line.

\subsubsection{Case: $\mathcal{C}_1$ smooth}\label{subsubsec:C1smooth} In that situation $\deg \omega_{\mathcal{C}_1}(h)=3$ so according to (\ref{eq:blabla2}) the map $\psi$ has two ordinary base points. The curve $\mathcal{C}_2$ has genus $-1$ and does not lie on a quadric; $\mathcal{C}_2$ is the disjoint union of a twisted cubic and a line, that is~$\psi$ belongs to~$\mathcal{E}_{12}$. Indeed suppose that $\psi\not\in\overline{\mathcal{E}_{12}}$, then~$\mathcal{C}_2$ is the union of two smooth conics $\Gamma_1$ and $\Gamma_2$ that do not intersect. Any $\Gamma_i$ is contained in a plane $\mathcal{P}_i$. Denote by $\ell$ the intersection $\mathcal{P}_1\cap\mathcal{P}_2$. As $\#\big(\ell\cap(\Gamma_1\cup\Gamma_2)\big)=4$, all the cubic surfaces that contain $\Gamma_1\cup\Gamma_2$ contain $\ell$. So $\ell\subset \mathcal{C}_2$: contradiction. 

\subsubsection{Case: $\mathcal{C}_1$ not smooth}\label{subsubsec:C1notsmooth} So $\mathfrak{p}_a(\mathcal{C}_1)\geq 1$, and 
\[
\mathfrak{p}_a(\mathcal{C}_2)=\deg \mathcal{C}_2-\deg \mathcal{C}_1+\mathfrak{p}_a(\mathcal{C}_1)=-1+\mathfrak{p}_a(\mathcal{C}_1)\geq 0.
\]
Since $\mathcal{C}_1$ is not in a plane, $\mathfrak{p}_a(\mathcal{C}_1)\leq 2$. Therefore we only have to distinguish the eventualities~$\mathfrak{p}_a(\mathcal{C}_1)=~1$ and $\mathfrak{p}_a(\mathcal{C}_1)=2$. Before looking at any of these eventualities let us introduce the set
\[
\mathscr{C}=\big\{\text{irreducible curves of $\mathbb{P}^3_\mathbb{C}$ of degree $5$ and geometric genus $0$}\big\}
\]
 
\smallskip

$\bullet$ Assume first that $\mathfrak{p}_a(\mathcal{C}_1)=1$. Then $\mathcal{O}_{\mathcal{C}_1}=\omega_{\mathcal{C}_1}$. We will denote by $\pi\colon\mathbb{P}^1_\mathbb{C}\to\mathcal{C}_1$ the norma\-lization of $\mathcal{C}_1$.

\begin{enumerate}
\item[$a_1)$] Suppose first that all the elements of $\Lambda_\psi$ are singular at $p\in\mathbb{P}^3_\mathbb{C}$. Denote by $L$ the $2$-dimensional vector space $\Lambda_\psi\cap\mathrm{H}^0\big((\mathcal{I}_p^2\cap\mathcal{I}_{\mathcal{C}_1})(3h)\big)$ defining $\mathcal{C}_1$ and $\mathcal{C}_2$. By the liaison sequence (\ref{eq:blabla2}) of Lemma~\ref{Lem:liaison} $\frac{\Lambda_\psi}{L}$ gives a vector subspace $u$ of $\mathrm{H}^0\big(\omega_{\mathcal{C}_1}(h)\big)=\mathrm{H}^0\big(\mathcal{O}_{\mathcal{C}_1}(h)\big)$ of dimension $2$. It induces a projection from $\mathcal{C}_1$ to $\vert u^\vee\vert$ that coincides with the restriction of~$\psi$ to $\mathcal{C}_1$; hence this projection has degree $1$. Moreover, via the identification $\mathrm{H}^0\big(\mathcal{O}_{\mathcal{C}_1}(h)\big)=\mathrm{H}^0\big(\pi^*\mathcal{O}_{\mathcal{C}_1}(h)\big)$, $u$ is included in the set $V_1$ of sections of $\mathcal{O}_{\mathbb{P}^1_\mathbb{C}}(5)$ whose base locus contains $\pi^{-1}(p)$; there are two other ordinary base points. 

We would like to show that $\mathcal{C}_2$ moves in an irreducible family. We will do this by deforming $\psi$ (and $\mathcal{C}_2$) while $\mathcal{C}_1$ is fixed. So, $p\in\mathbb{P}^3_\mathbb{C}$ being fixed, let us consider 
\[
\mathcal{R}_{p,1}=\big\{\mathcal{C}\in\mathscr{C}\,\vert\,\{p\}=\mathrm{Sing}\,\mathcal{C},\,\mathfrak{p}_a(\mathcal{C})=1\big\};
\]
the set $\mathcal{R}_{p,1}$ is an irreducible one. Remark that $\mathrm{h}^0\mathcal{I}_{\mathcal{C}}(3h)=\deg\mathcal{C}+\mathfrak{p}_a(\mathcal{C})-1=5$ and $\mathrm{h}^0\big((\mathcal{I}_p^2\cap\mathcal{I}_{\mathcal{C}})(3h))\big)=5-1=4$ because $\mathcal{C}$ has a double point at $p$ for all $\mathcal{C}$ in $\mathcal{R}_{p,1}$.

Let us denote by $F_1$ the set of $(\mathcal{C},L,u)\in\mathcal{R}_{1,p}\times\mathrm{H}^0\big((\mathcal{I}_p^2\cap\mathcal{I}_{\mathcal{C}})(3h)\big)\times V_1$ defined by
\begin{itemize}
\item $L\subset\mathrm{H}^0\big((\mathcal{I}_p^2\cap\mathcal{I}_{\mathcal{C}})(3h)\big)$ of dimension $2$ such that the residual of $\mathcal{C}$ in the complete intersection defined by $L$ has no common component with $\mathcal{C}$, and $\mathcal{C}$ is geometrically linked to a curve denoted by $\mathcal{C}_{2,L}$,
\smallskip
\item $u\subset V_1$ of dimension $2$ such that $\mathbb{P}^1_\mathbb{C}\dashrightarrow\vert u^\vee\vert$ has degree $1$. 
\end{itemize}
The set $F_1$ is irreducible since the choice of $\mathcal{C}$ is irreducible, and thus the choices of $L$ and~$u$ too.

If $(\mathcal{C},L,u)$ belongs to $F_1$, let us set 
\[
h_L\colon\mathrm{H}^0\big(\mathcal{I}_{\mathcal{C}_{2,L}}(3h)\big)\to\mathrm{H}^0\big(\omega_{\mathcal{C}}(h)\big) 
\]
(recall that $\frac{\mathrm{H}^0\big(\mathcal{I}_{\mathcal{C}_{2,L}}(3h)\big)}{L}\simeq\mathrm{H}^0\big(\omega_{\mathcal{C}}(h)\big)$). Consider the map
\[
\kappa_1\colon F_1\to \mathbb{G}\big(4;\mathrm{H}^0\big(\mathcal{O}_{\mathbb{P}^3_\mathbb{C}}(3)\big)\big),\qquad(\mathcal{C},L,u) \mapsto h_L^{-1}(u).
\]

By construction of $F_1$ if $\psi$ is birational, if all elements of $\Lambda_\psi$ are singular at $p$, and $\mathfrak{p}_a(\mathcal{C}_1)=1$, then $\Lambda_\psi$ is in the image of $\kappa_1$.
 
\begin{lem}
The general element of $\mathrm{im}\,\kappa_1$ coincides with $\Lambda_\psi$ for some birational map $\psi$ of $\mathcal{E}_{14}$.
\end{lem}

\begin{proof}
As $F_1$ is irreducible it is enough to show that $h_L^{-1}(u)$ is a birational system when $(\mathcal{C},L,u)$ is general in $F_1$. In that situation $\mathcal{C}_{2,L}$ is a curve of degree $4$, arithmetic genus $0$, singular at $p$, lying on a smooth quadric. Therefore $\mathcal{C}_{2,L}$ is reducible; more precisely it is the union of a twisted cubic and a line of this smooth quadric. All the elements of $h_L^{-1}(u)$ are cubic surfaces singular at $p$ because $\mathcal{C}_{2,L}$ has a double point at $p$, and the residual pencil $u\subset\mathrm{H}^0\big(\omega_{\mathcal{C}}(h)\big)$ vanishes at $p$ by definition of $F_1$. From definition of $u$, $h_L^{-1}(u)$ has two ordinary base points $p_1$ and $p_2$. Hence let $\mathcal{C}_1$ be the residual of $\mathcal{C}_{2,L}$ in the intersection of two general cubics of $h_L^{-1}(u)$. Then $\mathcal{C}_1$ is singular at $p$, $\psi_{\mathcal{C},L,u}$ has degree $5$ on $\mathcal{C}_1$, sends~$\mathcal{C}_1$ onto a line, is birational, and its base locus contains $p$, $p_1$, $p_2$.
\end{proof}

Let us remark that the previous irreducibility result asserts that the following example (belonging to family $\mathcal{E}_{18}$) that is not on a smooth quadric is nevertheless a deformation of elements of $\mathcal{E}_{14}$.

\begin{eg}
Let $\mathcal{C}_2$ be the union of a line doubled on a smooth quadric with two other lines, such that all these lines contain a same point $p$. Set 
\[
Q=z_0z_3-z_1z_2,\qquad \mathcal{I}_p=(z_0,z_1,z_2); 
\]
then $\mathcal{I}_{\mathcal{C}_2}=((z_2,z_0)^2 +(Q)) \cap (z_1,z_2) \cap (z_0-z_2,z_1-z_2)$. Now chose a double point of contact (note that the tangent cone must contain the tangent cone of $\mathcal{C}_2$): 
\[
\mathcal{I}_{\mathrm{dpc}}= (z_2^2z_3-z_0z_1z_3)+(z_0,z_1,z_2)^3, 
\]
and let $p_1$ and $p_2$ be two general points. Define $\mathcal{I}_{\psi}$ by $\mathcal{I}_{\mathcal{C}_2} \cap \mathcal{I}_{\mathrm{dpc}} \cap\mathcal{I}_{p_1} \cap \mathcal{I}_{p_2}$. So $I_{\psi}$ is the intersections of $I_{p_1}\cap I_{p_2}$ with
\begin{small}
\[
\hspace*{1cm}\mathcal{I}_{\mathcal{C}_2} \cap
\mathcal{I}_{\mathrm{dpc}} = \big(z_1z_2^2-z_2^3,\,z_0z_2^2-z_2^3,\,z_1^2z_2-z_2^3-z_0z_1z_3+z_2^2z_3,\,z_0z_1z_2-z_2^3,\,z_0^2z_2-z_2^3,\,z_0^2z_1-z_2^3\big).\]  
\end{small}
The tangent cone of $\mathcal{C}_2$ at $p$ has degree $4$ but the tangent cone of $\mathcal{C}_1\cup\mathcal{C}_2$ at $p$ has degree~$6$, so $\mathcal{C}_1$ belongs to $\mathcal{R}_{p,1}$.
\end{eg}

\item[$b_1)$] Suppose now that $\Lambda_\psi$ contains a smooth element at $p$. Then $p$ is a point of contact, all the cubic surfaces are tangent at $p$; $\mathcal{C}_2\subset Q$ is linked to a curve of degree $2$ and genus $-1$. In that case we have no restriction on the curves of genus $0$ and degree $4$ contrary to the previous case. Hence in general $Q$ is smooth, and $\mathcal{C}_2$ is a smooth rational curve on $Q$. Set $Q=z_0z_3-z_1z_2$, $\mathcal{I}_{\ell_1}=(z_0,z_1)$, and $\mathcal{I}_{\ell_2}=(z_2,z_3)$; one has 
\[
\mathcal{J}=\mathcal{I}_{\ell_1\cup\ell_2}=(z_0z_2,z_0z_3,z_1z_2,z_1z_3).
\]
Let $\mathcal{S}_0$ be the element of $\mathcal{J}$ given by 
\[
az_0z_2+bz_0z_3+cz_1z_3\qquad a,\, b,\, c\in\mathbb{C}[z_0,z_1,z_2,z_3]_1;
\]
one has $\mathcal{I}_{\mathcal{C}_2}=((\mathcal{S}_0,Q):\mathcal{J})=(Q,\mathcal{S}_0,\mathcal{S}_1,\mathcal{S}_2)$ with
\[
\mathcal{S}_1=z_0^2a+z_0z_1b+z_1^2c, \quad \mathcal{S}_2=z_2^2a+z_2z_3b+z_3^2c.
\]
The dimension of $\mathrm{H}^0\big(\mathcal{I}_{\mathcal{C}_2}(3h)\big)$ is $7$; indeed one has the following seven cubics: 
\[
\mathcal{I}_{\mathcal{C}_2}=\langle Qz_0,\, Qz_1,\, Qz_2,\, Qz_3, \,\mathcal{S}_0,\, \mathcal{S}_1,\, \mathcal{S}_2\rangle. 
\]
The map $\psi$ has no base point. Indeed $u=\frac{\Lambda_\psi}{\mathrm{H}^0\big(\mathcal{I}_{\mathcal{C}_1\cup\mathcal{C}_2}(3)\big)}$ is contained in the sections of~$\mathcal{O}_{\mathbb{P}^1_\mathbb{C}}(5)$ whose base locus contains $2\pi^{-1}(p)$; we thus already have an isomorphism between $\mathbb{P}^1_\mathbb{C}$ and $\vert u^\vee\vert$. The map~$\psi$ belongs to $\overline{\mathcal{E}_{23}}$. 
\end{enumerate}

$\bullet$ Suppose that $\mathfrak{p}_a(\mathcal{C}_1)=2$. Then $\mathfrak{p}_a(\mathcal{C}_2)=1$, $\mathcal{C}_1$ lies on a quadric and $\mathrm{h}^0\mathcal{I}_{\mathcal{C}_1}(3)=6$. We will still denote by $\pi\colon\mathbb{P}^1_\mathbb{C}\to\mathcal{C}_1$ the normalization of $\mathcal{C}_1$.

\begin{enumerate}
\item[$a_2)$] Assume first that $\mathcal{C}_1$ has a triple point $p$. The curve $\mathcal{C}_1$ is linked to a line by a complete intersection $(Q,\mathcal{S}_0)$ where $Q$ (resp. $\mathcal{S}_0$) is a cone (resp. a cubic) singular at $p$. We can write the normalization $\pi$ as follows $(\alpha^2A,\alpha\beta A,\beta^2A,B)$ with $A\in\mathbb{C}[\alpha,\beta]_3$, $B\in\mathbb{C}[\alpha,\beta]_5$, and $A$, $B$ without common factors. Then $Q=z_1^2-z_0z_2$, and $\mathrm{H}^0\big(\omega_{\mathcal{C}_1}(h)\big)$ can be identified with $\mathrm{H}^0\big(\mathcal{I}_\ell(2)\big)$, where $\mathcal{I}_\ell=(z_0,z_1)$. So $\mathrm{H}^0\big(\omega_{\mathcal{C}_1}(h)\big)$ is the $6$-dimensional subspace~$W$ of $\mathrm{H}^0\big(\mathcal{O}_{\mathbb{P}^1_\mathbb{C}}(6)\big)$ spanned by $(\alpha,\beta)\cdot(\alpha^2A,\alpha\beta A,\beta^2A,B)$. Let us consider the subspace $V_A=W\cap(A)$ of $W$. Let $L$ be the $2$-dimensional vector space $\Lambda_\psi\cap\mathrm{H}^0\big(\mathcal{I}_{\mathcal{C}_1}(3h)\big)$. Then $\frac{\Lambda_\psi}{L}$ gives a $2$-dimensional vector subspace $u$ of $V_A$. The restriction of~$\psi$ to $\mathcal{C}_1$ gives a birational map $\mathbb{P}^1_\mathbb{C}\dashrightarrow\vert u^\vee\vert$ induced by $u\subset V_A\subset\mathrm{H}^0\big(\mathcal{O}_{\mathbb{P}^1_\mathbb{C}}(6)\big)$. Furthermore~$\psi$ has two ordinary base points. We would like to show that in that case $\mathcal{C}_2$ moves in an irreducible family whose general element is the complete intersection of two quadrics. We thus fix a point $p\in\mathbb{P}^3_\mathbb{C}$ and introduce the irreducible set 
\[
\mathcal{R}_{p,2}=\big\{\mathcal{C}\in\mathscr{C}\,\vert\, \{p\}=\mathrm{Sing}\,\mathcal{C},\,\mathfrak{p}_a(\mathcal{C})=2\big\}.
\]
We define the set $F_2$ as the $(\mathcal{C},L,u)\in\mathcal{R}_{p,2}\times\mathrm{H}^0\big(\mathcal{I}_{\mathcal{C}}(3h)\big)\times V_A$ given by
\begin{itemize}
\item $L\subset\mathrm{H}^0\big(\mathcal{I}_{\mathcal{C}}(3h)\big)$ of dimension $2$ such that the residual of $\mathcal{C}$ in the complete intersection defined by $L$ has no common component with $\mathcal{C}$, and $\mathcal{C}$ is geometrically linked to a curve denoted by $\mathcal{C}_{2,L}$,
\smallskip
\item $u\subset V_A$ of dimension $2$ such that $\mathbb{P}^1_\mathbb{C}\dashrightarrow \vert u^{\vee}\vert$ is birational and whose base locus contains $\pi^{-1}(p)$.
\end{itemize}

Let us consider the map
\[
\kappa_2\colon F_2\to \mathbb{G}\big(4;\mathrm{H}^0\big(\mathcal{O}_{\mathbb{P}^3_\mathbb{C}}(3)\big)\big),\qquad(\mathcal{C},L,u) \mapsto h_L^{-1}(u).
\]

If $\psi$ is birational, if $\mathfrak{p}_a(\mathcal{C}_1)=2$, and $\mathcal{C}_1$ has a triple point then $\psi$ belongs to $\mathrm{im}\,\kappa_2$.

\begin{lem}
The general element of $\mathrm{im}\,\kappa_2$ coincides with $\Lambda_\psi$ for some birational map $\psi$ of $\mathcal{E}_{13}$.
\end{lem}

\begin{proof}
As $F_2$ is irreducible one can consider a general element of $F_2$, and then $\mathcal{C}_{2,L}$ is a curve of degree $4$, genus $1$ and is the complete intersection of two smooth quadrics. The map $\psi$ has two ordinary base points $p_1$, $p_2$, and belongs to $\mathcal{E}_{13}$. More precisely $\Lambda_\psi=\mathrm{H}^0\big((\mathcal{I}_{\mathcal{C}_{2,L}}\cdot\mathcal{I}_p\cap \mathcal{I}_{p_1}\cap\mathcal{I}_{p_2})(3)\big)$.
\end{proof}

Note that this irreducibility result asserts that the following example, where $\mathcal{C}_2$ is not a complete intersection of two quadrics is nevertheless a deformation of elements of $\mathcal{E}_{13}$.

\begin{eg}
Let $\mathcal{C}_2$ be the union of a plane cubic $\mathcal{C}_3$ singular at $p$ and a line $\ell$ containing~$p$ but not in the plane spanned by $\mathcal{C}_3$. For instance take $\mathcal{I}_p=(z_0,z_1,z_2)$, $\mathcal{I}_\ell=(z_1,z_2)$, $\mathcal{I}_{\mathcal{C}_3}=(z_1-z_0,(z_1-z_2)z_1z_3+z_0^3+z_1^3+z_2^3)$. Let $\mathcal{I}_{\textrm{dpc}}$ be a double point of contact at $p$. (As we have already chose $\mathcal{C}_2$, we must take a quadric cone containing the tangent cone to $\mathcal{C}_2$). For instance one can take: $I_{\textrm{dpc}}=(z_1^2-z_0z_2)+\mathcal{I}_p^3$, and let
\begin{eqnarray*}
\mathcal{J}=\mathcal{I}_{\mathcal{C}_2}\cap \mathcal{I}_{\textrm{dpc}}&=&\big(z_0z_2^2-z_1z_2^2,z_0z_1z_2-z_1^2z_2,\,z_0^2 z_2-z_1^2z_2,\,2 z_1^3+z_2^3+z_1^2z_3-z_0z_2z_3,\\
& &\hspace{0.5cm}2 z_0z_1^2+z_2^3+z_1^2z_3-z_0z_2z_3,\,2z_0^2z_1+z_2^3+z_1^2z_3-z_0z_2z_3\big)  
\end{eqnarray*}
chose two general points $p_1$ and $p_2$ and define by $\mathcal{I}_\psi$ the ideal generated by the $4$ cubics of $\mathcal{J}\cap \mathcal{I}_{p_1} \cap \mathcal{I}_{p_2}$. The tangent cone of $\mathcal{C}_2$ at $p$ has degree $3$, the tangent cone of $\mathcal{C}_1\cup\mathcal{C}_2$ at $p$ has degree $6$ (because $p$ is a double point of contact); hence $\mathcal{C}_1$ has also a triple point at $p$, and belongs to $\mathcal{R}_{p,2}$.
\end{eg}

\smallskip

\item[$b_2)$] Suppose now that $\mathcal{C}_1$ hasn't a triple point; $\mathcal{C}_1$ has thus two distinct double points. Fix two distinct points $p$ and $q$ in $\mathbb{P}^3_\mathbb{C}$, and set 
\[
\mathcal{R}_{p,q,2}=\big\{\mathcal{C}\in\mathscr{C}\,\vert\, \{p,\, q\}= \mathrm{Sing}\,\mathcal{C},\, \mathfrak{p}_a(\mathcal{C})=2\big\}.
\]
Let $V_3$ (resp. $V_4$) be the sections of $\mathcal{O}_{\mathbb{P}^1_\mathbb{C}}(7)$ whose base locus contains $\pi^{-1}(p)$ and $\pi^{-1}(q)$ (resp. $\pi^{-1}(p)$ and $2\pi^{-1}(q)$). The set $\mathcal{R}_{p,q,2}$ is irreducible. Remark that for all $\mathcal{C}$ in $\mathcal{R}_{p,q,2}$ one has 
\[
\mathrm{h}^0\mathcal{I}_{\mathcal{C}}(3h)=6,\qquad\mathrm{h}^0\big((\mathcal{I}_{\mathcal{C}}\cap\mathcal{I}_p^2)(3h)\big)=5,\qquad\mathrm{h}^0\big((\mathcal{I}_{\mathcal{C}}\cap\mathcal{I}_p^2\cap\mathcal{I}_q^2)(3h)\big)=4. 
\]

\begin{rem}
One cannot have two distinct points of contact. Assume by contradiction that there are two distinct points of contact $p$ and $q$. Denote by $\pi\colon\widetilde{\mathcal{C}_1}\to\mathcal{C}_1$ the normalization of $\mathcal{C}_1$. One would have $\pi^*\omega_{\mathcal{C}_1}(h)=\mathcal{O}_{\mathbb{P}^1_\mathbb{C}}(7)$ but the linear system induced by $\psi$ would contain in the base locus $2\pi^{-1}(p)+2\pi^{-1}(q)$ which is of length $8$: contradiction with the fact that $\psi(\mathcal{C}_1)$ is a line. 
\end{rem}

So one has the following alternative:
\begin{enumerate}
\item[$b_2)$\,i)] Either all the cubics of $\Lambda_\psi$ are singular at $p$ and $q$. One can then define the set $F_3$ of $(\mathcal{C},L,u)\in\mathcal{R}_{p,q,2}\times\mathrm{H}^0\big((\mathcal{I}_{\mathcal{C}}\cap\mathcal{I}_p^2\cap\mathcal{I}_q^2)(3h)\big)\times V_3$ given by 
\begin{itemize}
\item $L\subset\mathrm{H}^0\big((\mathcal{I}_{\mathcal{C}}\cap\mathcal{I}_p^2\cap\mathcal{I}_q^2)(3h)\big)$ of dimension $2$ such that the residual of $\mathcal{C}$ in the complete intersection defined by $L$ has no common component with $\mathcal{C}$;
\smallskip
\item $u\subset V_3$ of dimension $2$ such that $\mathcal{C}\dashrightarrow\vert u^{\vee}\vert$ has degree $1$.
\end{itemize}

Let us consider the map
\[
\kappa_3\colon F_3\to \mathbb{G}\big(4;\mathrm{H}^0\big(\mathcal{O}_{\mathbb{P}^3_\mathbb{C}}(3)\big)\big),\qquad(\mathcal{C},L,u) \mapsto h_L^{-1}(u).
\]

If $\psi$ is birational, if $\mathfrak{p}_a(\mathcal{C}_1)=2$, if $\mathcal{C}_1$ has two distinct double points at $p$ and $q$ and if all the cubics of $\Lambda_\psi$ are singular at $p$ and $q$, then $\Lambda_\psi$ belongs to $\mathrm{im}\,\kappa_3$.

\begin{lem}
The general element of $\mathrm{im}\,\kappa_3$ coincides with $\Lambda_\psi$ for some birational map~$\psi$ of $\mathcal{E}_{19}$.
\end{lem}

\begin{proof} 
As $F_3$ is irreducible one can consider a general element $(\mathcal{C},L,u)$ of $F_3$ and then $\mathcal{C}_{2,L}$ is a curve of degree $4$ and genus $1$, is singular at $p$ and $q$, lies on a smooth quadric, and is reducible: $\mathcal{C}_{2,L}$ is the union of a twisted cubic $\Gamma$ and the line $\ell=(pq)$. Moreover all the elements of $h_L^{-1}(u)$ are singular at $p$ and $q$ (by definition of $V_3$ and by the fact that $\mathcal{C}_{2,L}$ is singular at $p$ and $q$).
\end{proof}

In this situation as all the cubic surfaces are singular at $p$ and $q$, 
\[
h_L^{-1}(u)=\mathrm{H}^0\big((\mathcal{I}_\Gamma\cdot\mathcal{I}_\ell\cap\mathcal{I}_{p_1}\cap\mathcal{I}_{p_2})(3)\big)
\]
where $p_1$, $p_2$ are two ordinary base points; $\psi$ belongs to $\mathcal{E}_{19}$. 

\item[$b_2)$\,ii)] Or one of the cubic of $\Lambda_\psi$ is smooth at (for instance) $q$.
Let us introduce the set $F_4$ of pairs $(\mathcal{C},L)\in\mathcal{R}_{p,q,2}\times\mathrm{H}^0\big((\mathcal{I}_{\mathcal{C}}\cap\mathcal{I}_p^2)(3h)\big)$ satisfying: $L\subset\mathrm{H}^0\big((\mathcal{I}_{\mathcal{C}}\cap\mathcal{I}_p^2)(3h)\big)$ of dimension $2$ such that the residual of $\mathcal{C}$ in the complete intersection defined by $L$ has no common component with $\mathcal{C}$.

Let us consider the map
\[
\kappa_4\colon F_4\to \mathbb{G}\big(4;\mathrm{H}^0\big(\mathcal{O}_{\mathbb{P}^3_\mathbb{C}}(3)\big)\big),\qquad(\mathcal{C},L) \mapsto h_L^{-1}(V_4);
\]
note that $\dim V_4=2$.

If $\psi$ is birational, if $\mathfrak{p}_a(\mathcal{C}_1)=2$, $\mathcal{C}_1$ hasn't a triple point and one of the cubic of $\Lambda_\psi$ is smooth at (for instance) $q$, then $\Lambda_\psi$ belongs to $\mathrm{im}\,\kappa_4$.

\begin{lem}
The general element of $\mathrm{im}\,\kappa_4$ coincides with $\Lambda_\psi$ for some birational map~$\psi$ of $\mathcal{E}_{24}$.
\end{lem}

\begin{proof}
As $F_4$ is irreducible one can consider a general element of $F_4$, and then $\mathcal{C}_{2,L}$ is a curve of degree $4$, genus $1$, singular at $p$, and is the complete intersection of two quadrics. The map $\psi$ has no base point and belongs to $\mathcal{E}_{24}$.
\end{proof}
\end{enumerate}
\end{enumerate}

\subsection{Irreducible components}

The following statement, and Theorems \ref{thm:comp33} and \ref{thm:comp34} imply Theorem \ref{thmA}.

\begin{thm}\label{thm:comp35}
One has the inclusions: $\mathcal{E}_{14}\subset\overline{\mathcal{E}_{12}}$, $\mathcal{E}_{24}\subset\overline{\mathcal{E}_{23}}$, and $\mathcal{E}_{19}\subset\overline{\mathcal{E}_{12}}$.

The set $\mathrm{Bir}_{3,5}(\mathbb{P}^3_\mathbb{C})$ has four irreducible components: $\mathcal{E}_{12}$, $\mathcal{E}_{13}$, $\mathcal{E}_{23}$, and $\mathcal{E}_{27}=\mathfrak{ruled}_{3,5}$.
\end{thm}

\begin{proof}
Let us first prove that $\mathcal{E}_{14}\subset\overline{\mathcal{E}_{12}}$. If $\psi$ belongs to $\mathcal{E}_{12}$, or to $\mathcal{E}_{14}$ the curve $\mathcal{C}_2$ is the union of a line $\ell$ and a twisted cubic $\Gamma$ such that $\mathrm{length}\,(\ell\cap\Gamma)\leq 1$. Let $\mathcal{I}_\ell$ (resp. $\mathcal{I}_\Gamma$) be the ideal of $\ell$ (resp.~$\Gamma$). We have $\mathcal{I}_\psi\subset\mathcal{I}_\ell\cap\mathcal{I}_\Gamma$. If $\psi$ belongs to $\mathcal{E}_{12}$, then $\ell\cap\Gamma=\emptyset$, and $\mathcal{I}_\ell\cap\mathcal{I}_\Gamma=\mathcal{I}_\ell\cdot\mathcal{I}_\Gamma$. And if~$\psi$ is in~$\mathcal{E}_{14}$, then all the cubics are singular at $p=\ell\cap\Gamma$ so $\mathcal{I}_\psi$ is again in $\mathcal{I}_\ell\cdot\mathcal{I}_\Gamma$.

\smallskip

Prove now that $\mathcal{E}_{24}\subset\overline{\mathcal{E}_{23}}$. Consider a general element $\psi$ of $\mathcal{E}_{24}$; the curve $\mathcal{C}_2$ is the complete intersection of a quadric $Q'=az_2+bz_0+cz_1$ passing through the double point $p$ and a cone $Q_0=z_1z_2-z_0^2$. Furthermore all the cubics of $\mathcal{I}_\psi$ are singular at $p$, and $\mathcal{I}_\psi\subset\mathcal{J}'_0=(Q_0,z_0Q',z_1Q',z_2Q')$. Let $\mathfrak{ct}_q$ be the ideal of the point of contact $q$; one has $\mathfrak{ct}_q=\mathcal{I}_q^2+(H_q)$ where $H_q$ is a plane passing through $q$. Denote by $\mathcal{I}_0$ the intersection of $\mathcal{J}'_0$ and $\mathfrak{ct}_q$. Set
\begin{align*}
& Z_0=z_0+tz_3,&& Z_1=z_1, && Z_2=z_2, && Z_3=z_0-tz_3, 
\end{align*}
\begin{align*}
& Q_t=Z_1Z_2-Z_0Z_3, && \mathcal{S}_0=aZ_0Z_2+bZ_0Z_3+cZ_1Z_3,\\
& \mathcal{S}_1=aZ_0^2+bZ_0Z_1+cZ_1^2, && \mathcal{S}_2=aZ_2^2+bZ_2Z_3+cZ_3^2.
\end{align*}
Hence $\mathcal{J}_t=(Q_t,\mathcal{S}_0,\mathcal{S}_1,\mathcal{S}_2)$ is the ideal of a rational quartic if $t\not=0$ (\emph{cf.} the equations in \S \ref{subsubsec:C1smooth} $b_1)$). The ideal $\mathcal{I}_t=\mathcal{J}_t\cap\mathfrak{ct}_q$ is the ideal $\mathcal{I}_\psi$ of $\psi\in\mathcal{E}_{23}$. Remark that if $t=0$, then 
\[
\mathcal{J}_0=(Q_0,z_0Q',az_0^2+bz_0z_1+cz_1^2,az_2^2+bz_0z_2+cz_0^2)
\]
but $az_0^2+bz_0z_1+cz_1^2=z_1Q'$ modulo $Q$, and $az_2^2+bz_0z_2+cz_0^2=z_2Q'$ modulo $Q$, that is $\mathcal{J}'_0=\mathcal{J}_0$. Therefore $\mathcal{I}_t$ tends to $\mathcal{I}_0$ as $t$ tends to $0$.

\smallskip
   
The inclusion $\mathcal{E}_{19}\subset\overline{\mathcal{E}_{12}}$ follows from $\Lambda_\psi=\mathrm{H}^0\big((\mathcal{I}_\ell\cdot\mathcal{I}_\Gamma\cap\mathcal{I}_{p_1}\cap\mathcal{I}_{p_2})(3)\big)$ found in $b_2)$\,i).

\smallskip

Note that $\mathcal{E}_{12}\not\subset\overline{\mathcal{E}_{13}}$ (resp. $\mathcal{E}_{12}\not\subset\overline{\mathcal{E}_{23}}$): if $\psi$ is in $\mathcal{E}_{12}$ then the associated $\mathcal{C}_2$ does not lie on a quadric whereas if $\psi$ belongs to $\mathcal{E}_{13}$ (resp. $\mathcal{E}_{23}$) then $\mathcal{C}_2$ lies on two quadrics (resp. one quadric). Conversely $\mathcal{E}_{13} \not\subset\overline{\mathcal{E}_{12}}$ (resp. $\mathcal{E}_{23}\not\subset\overline{\mathcal{E}_{12}}$): if $\psi$ is an element of $\mathcal{E}_{13}$ (resp. $\mathcal{E}_{23}$), then $\mathcal{C}_2$ is smooth and irreducible whereas the associated $\mathcal{C}_2$ of a general element of $\mathcal{E}_{12}$ is the disjoint union of a twisted cubic and a line.

\smallskip

Let us now justify that $\mathcal{E}_{23}\not\subset\overline{\mathcal{E}_{13}}$: the linear system of an element of $\mathcal{E}_{23}$ has a smooth surface whereas the linear system of an element of $\mathcal{E}_{13}$ does not. Conversely $\mathcal{E}_{13}\not\subset\overline{\mathcal{E}_{23}}$; indeed $\mathrm{h}^0\mathcal{I}_{\mathcal{C}_2}(3h)=6$ for a birational map of $\mathcal{E}_{13}$ and $\mathrm{h}^0\mathcal{I}_{\mathcal{C}_2}(3h)=7$ for a birational map of $\mathcal{E}_{23}$. 
\end{proof}

In bidegree $(3,5)$ the description of $\mathrm{Bir}_{3,5,\mathfrak{p}_2}(\mathbb{P}^3_\mathbb{C})$ is very different from those of smaller bidegrees. Let us now prove Theorem \ref{thmC}.

\begin{thm}
The set $\mathrm{Bir}_{3,5,\mathfrak{p}_2}(\mathbb{P}^3_\mathbb{C})$ is empty as soon as $\mathfrak{p}_2\not\in\{-1,\,0,\,1\}$ and 
\begin{itemize}
\item[$\bullet$] if $\mathfrak{p}_2=-1$, then $\mathrm{Bir}_{3,5,\mathfrak{p}_2}(\mathbb{P}^3_\mathbb{C})$ is non-empty, and irreducible;

\item[$\bullet$] if $\mathfrak{p}_2=0$, then $\mathrm{Bir}_{3,5,\mathfrak{p}_2}(\mathbb{P}^3_\mathbb{C})$ is non-empty, and has two irreducible components: one formed by the birational maps of $\mathcal{E}_{14}$, and the other one by the elements of $\mathcal{E}_{23}$;

\item[$\bullet$] if $\mathfrak{p}_2=1$, then $\mathrm{Bir}_{3,5,\mathfrak{p}_2}(\mathbb{P}^3_\mathbb{C})$ is non-empty, and has three irreducible components: one formed by the birational maps of $\mathcal{E}_{13}$, a second one formed by the birational maps of $\mathcal{E}_{19}$, and a third one by the elements of $\mathcal{E}_{24}$.
\end{itemize}
\end{thm}

\begin{proof}
$\bullet$ Assume $\mathfrak{p}_2=-1$. In that case only one family appears : $\mathcal{E}_{12}$ (\emph{see} \S\,\ref{subsubsec:C1smooth}), and accor\-ding to Theorem \ref{thm:comp35} the family $\mathcal{E}_{12}$ is already an irreducible component of $\mathrm{Bir}_{3,5}(\mathbb{P}^3_\mathbb{C})$ so an irreducible component of $\mathrm{Bir}_{3,5,-1}(\mathbb{P}^3_\mathbb{C})$.

\medskip

$\bullet$ Suppose $\mathfrak{p}_2=0$. We found two families : $\mathcal{E}_{14}$ (case $a_1$) of \S\,\ref{subsubsec:C1smooth}), and $\mathcal{E}_{23}$ (case $b_1$) of \S\,\ref{subsubsec:C1notsmooth}). Note that for $\psi$ general in $\mathcal{E}_{23}$ $\Lambda_\psi$ contains smooth cubics whereas all cubics of $\Lambda_\psi$ are singular as soon as $\psi$ belongs to $\mathcal{E}_{14}$. Hence $\mathcal{E}_{23}\not\subset\overline{\mathcal{E}_{14}}$.

Take a general element of $\mathcal{E}_{14}$; it hasn't a base scheme of dimension $0$, connected and of length $\geq 3$ whereas elements of $\overline{\mathcal{E}_{23}}$ have. Therefore $\mathcal{E}_{14}\not\subset\overline{\mathcal{E}_{23}}$.

\medskip

$\bullet$ Assume last that $\mathfrak{p}_2=1$. Our study gives three families: $\mathcal{E}_{13}$, $\mathcal{E}_{19}$ and $\mathcal{E}_{24}$ (cases $a_2$), $b_2$)i) and $b_2$)ii) of \S\,\ref{subsubsec:C1notsmooth}). The general element of $\mathcal{E}_{19}$ has two double points whereas a general element of~$\mathcal{E}_{13}$ (resp. $\mathcal{E}_{24}$) has only one; thus $\mathcal{E}_{19}\not\subset\overline{\mathcal{E}_{13}}$ and $\mathcal{E}_{19}\not\subset\overline{\mathcal{E}_{24}}$.

Take a general element in $\mathcal{E}_{13}$; its base locus is a smooth curve. On the contrary if $\psi$ belongs to~$\mathcal{E}_{19}$ (resp. $\mathcal{E}_{24}$), then the base locus of $\psi$ is a singular curve. Thus $\mathcal{E}_{13}\not\subset\overline{\mathcal{E}_{19}}$ (resp. $\mathcal{E}_{13}\not\subset\overline{\mathcal{E}_{24}}$).

If $\psi$ is a general element of $\mathcal{E}_{24}$ its base locus is an irreducible curve that is not the case if~$\psi~\in~\mathcal{E}_{19}$ so $\mathcal{E}_{24}\not\subset\overline{\mathcal{E}_{19}}$.

Let us now consider a general element of $\mathcal{E}_{24}$, the tangent plane at all cubic surfaces at the point of contact doesn't contain the double point $p$; hence if we denote by $Q_1$ and $Q_2$ the quadrics containing $\mathcal{C}_2$ there doesn't exist a plane $h$ passing through $p$ such that $(hQ_1,hQ_2)\subset\Lambda_\psi$. But if we take $\psi$ in $\mathcal{E}_{13}$ then $\Lambda_\psi=\mathrm{H}^0\big((\mathcal{I}_{\mathcal{C}_2}\cdot\mathcal{I}_p\cap\mathcal{I}_{p_1}\cap\mathcal{I}_{p_2})(3)\big)$ with $p_1$, $p_2$ two ordinary base points, and~$p$ the triple point lying on $\mathcal{C}_1$. If $h$ is the plane passing through $p$, $p_1$ and $p_2$, if $\mathcal{I}_{\mathcal{C}_2}=(Q_1,Q_2)$, then $(hQ_1,hQ_2)\subset\Lambda_\psi$. Thus $\mathcal{E}_{24}\not\subset\overline{\mathcal{E}_{13}}$.
\end{proof}

\section{Relations with \textsc{Hudson}'s invariants}\label{Sec:singcubsurf}

To prove the birationality of a linear system of cubics, the local properties of $\mathcal{C}_1$ and $\mathcal{C}_2$ are required. For instance to apply Lemma \ref{lem:bir} one needs to understand the support of $\mathcal{C}_1\cup\mathcal{C}_2$ and the local intersection of $\mathcal{C}_1$ with a general element of $\Lambda_\psi$ at any point of $\mathcal{C}_1\cup\mathcal{C}_2$. So in the following table we make a schematic picture of the tangent cone of $\mathcal{C}_1\cup\mathcal{C}_2$ at one of its singular point in the different cases considered by \textsc{Hudson}. Let us note that the degree of the tangent cone of $\mathcal{C}_1\cup\mathcal{C}_2$ at a point of $\mathcal{C}_1\cup\mathcal{C}_2$ varies from $1$ to $6$. In particular if the linear system has a double point (resp. a double point of contact), then it is a complete intersection of two quadric cones (resp. of one quadric cone and one cubic cone). We draw pictures only when the quadric cone is irreducible. If the linear system has a binode, the tangent cone of $\mathcal{C}_1\cup\mathcal{C}_2$ has degree $5$; more precisely for a binode at $p=(z_0,z_1,z_2)$ whose fixed plane is $z_0$, {\it i.e.} $\mathcal{I}_\psi\subset\mathcal{I}_p\cdot(z_0)$, then the ideal of the tangent cone of $\mathcal{C}_1\cup\mathcal{C}_2$ at $p$ is $(z_0z_1,z_0z_2,P)$ where $P$ denotes an element of $\mathbb{C}[z_1,z_2]_4$. In our pictures the marked plane of the binode is vertical.

\smallskip

Convention. If the point is black (resp. white) then $\mathcal{C}_2$ does not pass (resp. passes through) through the point. For all cases mentioned in the paper we precise $(\widetilde{d}_1,\widetilde{d}_2)$ where $\widetilde{d}_i$ is the degree of the tangent cone of~$\mathcal{C}_i$ at $p$.

\smallskip

Let us mention that this table in which we propose local illustrations could help the reader to visualize the different examples but the proofs are not based on it.

\begin{landscape}
\begin{center}
  \begin{tabular}{ | c | c | c | c | c | c | c |}
   \hline    
     \multicolumn{7}{ |c| }{D.p. of contact} \\
     \hline
      \begin{pspicture}(-1,-1)(1,1)\psdots[dotstyle=*](-0.25,0)\psdots[dotstyle=*](0.18,-0.30)\psdots[dotstyle=*](-0.15,0.35)\psdots[dotstyle=*](0.25,0)\psdots[dotstyle=*](0,0.5)\psdots[dotstyle=*](0,-0.5)\psellipse(0,0)(0.25,0.5)\end{pspicture}& \begin{pspicture}(-1,-1)(1,1)\psdots[dotstyle=o](-0.25,0)\psdots[dotstyle=*](0.18,-0.30)\psdots[dotstyle=*](-0.15,0.35)\psdots[dotstyle=*](0.25,0)\psdots[dotstyle=*](0,0.5)\psdots[dotstyle=*](0,-0.5)\psellipse(0,0)(0.25,0.5)\end{pspicture} & \begin{pspicture}(-1,-1)(1,1)\psdots[dotstyle=o](-0.25,0)\psdots[dotstyle=o](0.18,-0.30)\psdots[dotstyle=*](-0.15,0.35)\psdots[dotstyle=*](0.25,0)\psdots[dotstyle=*](0,0.5)\psdots[dotstyle=*](0,-0.5)\psellipse(0,0)(0.25,0.5)\end{pspicture} & \begin{pspicture}(-1,-1)(1,1)\psdots[dotstyle=o](-0.25,0)\psdots[dotstyle=o](0.18,-0.30)\psdots[dotstyle=o](-0.15,0.35)\psdots[dotstyle=*](0.25,0)\psdots[dotstyle=*](0,0.5)\psdots[dotstyle=*](0,-0.5)\psellipse(0,0)(0.25,0.5)\end{pspicture} & \begin{pspicture}(-1,-1)(1,1)\psdots[dotstyle=o](-0.25,0)\psdots[dotstyle=o](0.18,-0.30)\psdots[dotstyle=o](-0.15,0.35)\psdots[dotstyle=o](0.25,0)\psdots[dotstyle=*](0,0.5)\psdots[dotstyle=*](0,-0.5)\psellipse(0,0)(0.25,0.5)\end{pspicture}& \begin{pspicture}(-1,-1)(1,1)\psdots[dotstyle=o](-0.25,0)\psdots[dotstyle=o](0.18,-0.30)\psdots[dotstyle=o](-0.15,0.35)\psdots[dotstyle=o](0.25,0)\psdots[dotstyle=o](0,0.5)\psdots[dotstyle=*](0,-0.5)\psellipse(0,0)(0.25,0.5)\end{pspicture}& \begin{pspicture}(-1,-1)(1,1)\psdots[dotstyle=o](-0.25,0)\psdots[dotstyle=o](0.18,-0.30)\psdots[dotstyle=o](-0.15,0.35)\psdots[dotstyle=o](0.25,0)\psdots[dotstyle=o](0,0.5)\psdots[dotstyle=o](0,-0.5)\psellipse(0,0)(0.25,0.5)\end{pspicture}  \\ 
      &  &  &  & $(2,4)$ &  &   \\ 
   \hline
 \multicolumn{7}{ |c| }{binode} \\
     \hline    
 \begin{pspicture}(-0.25,-0.25)(0.75,1)\psdots[dotstyle=*](0,0)(0,0.25)(0,0.5)(0,0.75)(0.25,0.75)\end{pspicture} 
                  & \begin{pspicture}(-0.25,-0.25)(0.75,1)\psdots[dotstyle=*](0,0)(0,0.25)(0,0.5)(0,0.75)
                    \psdots[dotstyle=o](0.25,0.75)\end{pspicture}  
                  & \begin{pspicture}(-0.25,-0.25)(0.75,1)\psdots[dotstyle=*](0.25,0.75)(0,0.25)(0,0.5)(0,0.75)
                    \psdots[dotstyle=o](0,0)\end{pspicture}
                  & \begin{pspicture}(-0.25,-0.25)(0.75,1)\psdots[dotstyle=*](0,0.25)(0,0.5)(0,0.75)
                    \psdots[dotstyle=o](0,0)(0.25,0.75)\end{pspicture}  
                  & \begin{pspicture}(-0.25,-0.25)(0.75,1)\psdots[dotstyle=*](0.25,0.75)(0,0.5)(0,0.75)
                    \psdots[dotstyle=o](0,0)(0,0.25)\end{pspicture}  & &\\    
                   &  & &  &  & &\\ 
     \hline                   
\begin{pspicture}(-0.25,-0.25)(0.75,1)\psdots[dotstyle=*](0,0.5)(0,0.75)
                    \psdots[dotstyle=o](0,0)(0,0.25)(0.25,0.75)\end{pspicture}
                  & \begin{pspicture}(-0.25,-0.25)(0.75,1)\psdots[dotstyle=*](0.25,0.75)(0,0.75)
                    \psdots[dotstyle=o](0,0)(0,0.25)(0,0.5)\end{pspicture}
                 & \begin{pspicture}(-0.25,-0.25)(0.75,1)\psdots[dotstyle=*](0,0)
                    \psdots[dotstyle=o](0,0.25)(0,0.5)(0.25,0.75)(0,0.75)\end{pspicture}
                 & \begin{pspicture}(-0.25,-0.25)(0.75,1)\psdots[dotstyle=*](0.25,0.75)
                    \psdots[dotstyle=o](0,0)(0,0.25)(0,0.5)(0,0.75)\end{pspicture}
                 & \begin{pspicture}(-0.25,-0.25)(0.75,1)\psdots[dotstyle=o](0,0)(0,0.25)(0,0.5)(0,0.75)(0.25,0.75)\end{pspicture} & & \\    
                $(2,3)$ & $(2,3)'$ &  & $(1,4)$ &  & & \\ 
     \hline
 \multicolumn{7}{ |c| }{D.p.'s} \\
     \hline 
     \begin{pspicture}(-0.25,-0.25)(0.75,1)\psdots[dotstyle=*](0,0)(0.5,0)(0,0.5)(0.5,0.5)\end{pspicture} 
     & \begin{pspicture}(-0.25,-0.25)(0.75,1)\psdots[dotstyle=*](0,0)(0.5,0)(0,0.5)\psdots[dotstyle=o](0.5,0.5)\end{pspicture}
     & \begin{pspicture}(-0.25,-0.25)(0.75,1)\psdots[dotstyle=*](0,0)(0.5,0)\psdots[dotstyle=o](0,0.5)(0.5,0.5)\end{pspicture}
     & \begin{pspicture}(-0.25,-0.25)(0.75,1)\psdots[dotstyle=*](0,0)\psdots[dotstyle=o](0.5,0)(0,0.5)(0.5,0.5)\end{pspicture}
     & \begin{pspicture}(-0.25,-0.25)(0.75,1)\psdots[dotstyle=o](0,0)(0.5,0)(0,0.5)(0.5,0.5)\end{pspicture} & & \\ 
     & & $(2,2)$ &  &  & & \\ 
     \hline
 \multicolumn{7}{ |c| }{pt of osculation} \\
     \hline
     \begin{pspicture}(-0.25,-0.25)(0.75,1)\psdots[dotstyle=*](0.25,0)(0.25,0.25)(0.25,0.5)\end{pspicture} 
    & \begin{pspicture}(-0.25,-0.25)(0.75,1)\psdots[dotstyle=*](0.25,0)(0.25,0.25)\psdots[dotstyle=o](0.25,0.5)\end{pspicture} 
      & \begin{pspicture}(-0.25,-0.25)(0.75,1)\psdots[dotstyle=*](0.25,0)\psdots[dotstyle=o](0.25,0.25)(0.25,0.5)\end{pspicture} 
    & \begin{pspicture}(-0.25,-0.25)(0.75,1)\psdots[dotstyle=o](0.25,0)(0.25,0.25)(0.25,0.5)\end{pspicture}& & & \\
    \hline
 \multicolumn{7}{ |c| }{pt of contact} \\
     \hline
     \begin{pspicture}(-0.25,-0.25)(0.75,1)\psdots[dotstyle=*](0.25,0)(0.25,0.25)\end{pspicture} 
    & \begin{pspicture}(-0.25,-0.25)(0.75,1)\psdots[dotstyle=*](0.25,0)\psdots[dotstyle=o](0.25,0.25)\end{pspicture} 
      &  \begin{pspicture}(-0.25,-0.25)(0.75,1)\psdots[dotstyle=o](0.25,0)(0.25,0.25)\end{pspicture}  
    & 
& & & \\
    \hline
  \end{tabular}
\end{center}
\begin{center}
\text{\texttt{Table} $1$}
\end{center}
\end{landscape}

\appendix

\section{\textsc{Hudson}'s Table}\label{Sec:hudsontable}

In this appendix we give a reproduction of what \textsc{Hudson} called ``Cubic Space Transformations''. The first (resp. second, resp. third, resp. fourth) table concerns birational maps of bidegrees $(3,2)$, $(3,3)$ and $(3,4)$ (resp. $(3,5)$, resp. $(3,6)$, resp. $(3,7)$, $(3,8)$ and $(3,9)$). 

\begin{landscape}
\begin{tabular}{|c|c|c|c|c|c|c|c|l|l|}
\hline
 & & & & & & & & & \\ 
 number& degrees& D.p. of & binode & D. p.'s & pt of & pt of & ordinary & $F$-curves & Remarks\\
& & contact & & & osculation & contact & pts &  & \\
 & & & & & & & & & \\ 
 \hline 
 $1$ & $3$--$2$ & $\cdot$& $\cdot$& $\cdot$& $\cdot$& $\cdot$& $\cdot$ & $l^2$, $l_1$, $l_2$, $l_3$ & {\small $3$ generators meet}\\ 
  & & & & & & & & & {\small double line}\\ 
 \hline
$2$ & $3$--$3$ & $\cdot$ & $\cdot$ & $\cdot$ & $\cdot$ & $\cdot$ & $\cdot$ & {\small $\omega_6$ (genus $3$)} &\\ 
$3$ & & $\cdot$ & $\cdot$ & $1$ & $\cdot$ & $\cdot$ & $\cdot$ & {\small $\omega_6\equiv O^2$ (genus $3$)} &\\ 
$4$ & & $1$ & $\cdot$ & $\cdot$ & $\cdot$ & $\cdot$ & $\cdot$ &{\small $\omega_6\equiv O^4$ (rational)}&\\ 
$5$ & & $\cdot$& $\cdot$& $\cdot$& $\cdot$& $\cdot$& $2$ & {\small $l^2$, $l_1$, $l_2$}& {\small $2$ generators meet}\\ 
  & & & & & & & & & {\small double line}\\ 
\hline
$6$ & $3$--$4$ & $\cdot$ & $\cdot$ & $\cdot$ & $\cdot$ & $\cdot$ & $1$ & {\small $\omega_5$ (genus $1$)}&\\ 
$7$ & & $\cdot$ & $\cdot$ & $1$ & $\cdot$ & $\cdot$ & $1$ & {\small $\omega_5\equiv O_1^2$ (genus $1$)}&\\ 
$8$ & & $\cdot$ & $1$ & $\cdot$ & $\cdot$ & $\cdot$ & $1$ & {\small $\omega_5\equiv O_1^3(2)$ (rational)} &\\ 
$9$ & & $1$ & $\cdot$ & $\cdot$ & $\cdot$ & $\cdot$ & $1$ & {\small $\omega_3\equiv O_1^2$, $l_1\equiv O_1$, $l_2\equiv O_1$} &\\ 
$10$ & & $1$ & $1$ & $\cdot$ & $\cdot$ & $\cdot$ & $1$ & {\small $\omega_2\equiv O_1O_2$, $l\equiv O_1O_2$ (osculation)}& {\small $(\phi)$ touch plane}\\
  & & & & & & & & & {\small along $l$}\\  
$11$ & & $\cdot$ & $\cdot$ & $\cdot$ & $\cdot$ & $\cdot$ & $4$ & {\small $l^2$, $l_1$} & {\small generator meets}\\ 
  & & & & & & & & & {\small double line}\\  
\hline
\end{tabular}
\end{landscape}

\begin{landscape}
\bigskip
\begin{tabular}{|c|c|c|c|c|c|c|l|l|}
\hline
 & & & & & & & & \\ 
 number& D.p. of & binode & D. p.'s & pt of & pt of & ordinary & $F$-curves & Remarks\\
& contact & & & osculation & contact & pts &  & \\
 & & & & & & & & \\ 
 \hline 
$12$ & $\cdot$ & $\cdot$ & $\cdot$ & $\cdot$ & $\cdot$ & $2$ & {\small $\omega_3$ (rational), $l$}&\\ 
$13$ & $\cdot$ & $\cdot$ & $1$ & $\cdot$ & $\cdot$ & $2$ & {\small $\omega_4\equiv O_1$ (genus 1)}&\\ 
$14$ & $\cdot$ & $\cdot$ & $1$ & $\cdot$ & $\cdot$ & $2$ & {\small $\omega_3\equiv O_1$ (rational), $l\equiv O_1$}&\\ 
$15$ & $\cdot$ & $1$ & $\cdot$ & $\cdot$ & $\cdot$ & $2$ & {\small $\omega_4\equiv O_1^2(2)$}&\\ 
$16$ & $\cdot$ & $1$ & $\cdot$ & $\cdot$ & $\cdot$ & $2$ & {\small $\omega_2\equiv O_1(1)$, $l_1\equiv O_1(1)$, $l_2\equiv O_1$} &\\ 
$17$ & $1$ & $\cdot$ & $\cdot$ & $\cdot$ & $\cdot$ & $2$ &{\small $\omega_3\equiv O_1^2$, $l_1\equiv O_1$}&\\ 
$18$ & $1$ & $\cdot$ & $\cdot$ & $\cdot$ & $\cdot$ & $2$ &{\small $l\equiv O_1$ (contact), $l_1\equiv O_1$, $l_2\equiv O_1$}& {\small $(\phi)$ touch quadric}\\ 
$19$ & $\cdot$ & $\cdot$ & $2$ & $\cdot$ & $\cdot$ & $2$ & {\small $\omega_3\equiv O_1O_2$ (rational), $l\equiv O_1O_2$}&\\ 
$20$ & $\cdot$ & $1$ & $1$ & $\cdot$ & $\cdot$ & $2$ & {\small $\omega_2\equiv O_1(1)O_2$, $l_1\equiv O_1O_2$, $l_2\equiv O_1(1)$}&\\ 
$21$ & $1$ & $\cdot$ & $1$ & $\cdot$ & $\cdot$ & $2$ & {\small $l\equiv O_1O_2$ (contact), $l_1\equiv O_1$, $l_2\equiv O_1$}& {\small $(\phi)$ touch plane}\\ 
$22$ & $1$ & $1$ & $\cdot$ & $\cdot$ & $\cdot$ & $2$ & {\small $l\equiv O_1O_2(1)$ (osculation), $l_1\equiv O_1$} & {\small $(\phi)$ touch plane}\\ 
$23$ & $\cdot$ & $\cdot$ & $\cdot$ & $\cdot$ & $1$ & $\cdot$& {\small $\omega_4$ (rational)}&\\ 
$24$ & $\cdot$ & $\cdot$ & $1$ & $\cdot$ & $1$ & $\cdot$ & {\small $\omega_4\equiv O_1^2$}&\\ 
$25$ & $\cdot$ & $1$ & $\cdot$ & $\cdot$ & $1$ & $\cdot$ & {\small $\omega_3\equiv O_1^2(1)$, $l\equiv O_1(1)$}&\\ 
$26$ & $1$ & $\cdot$ & $\cdot$ & $\cdot$ & $1$ & $\cdot$ & {\small $l_1\equiv O_1$, $l_2\equiv O_1$, $l_3\equiv O_1$, $l_4\equiv O_1$}&\\ 
$27$ & $\cdot$ & $\cdot$ & $\cdot$ & $\cdot$ & $\cdot$ & $6$ & {\small $l^2$}&\\  
\hline
\end{tabular}
\begin{center}
\text{Cubic Space Transformations of bidegree $(3,5)$}
\end{center}
\end{landscape}

\begin{landscape}
\begin{tabular}{|c|c|c|c|c|c|c|l|l|}
\hline
 & & & & & & & &\\ 
 number& D.p. of & binode & D. p.'s & pt of & pt of & ordinary & $F$-curves & Remarks\\
& contact & & & osculation & contact & pts &  & \\
 & & & & & & & &\\ 
\hline
$28$ & $\cdot$  & $\cdot$ & $\cdot$ & $\cdot$ & $\cdot$ & $3$ & {\small $l$ (contact), $l_1$} &\\ 
$29$ & $\cdot$ & $\cdot$ & $1$ & $\cdot$ & $\cdot$ & $3$ & {\small $\omega_3$ (plane, genus $1$)} &\\ 
$30$ & $\cdot$ & $\cdot$ & $1$ & $\cdot$ & $\cdot$ & $3$ & {\small $\omega_2$, $l\equiv O_1$}&\\  
$31$ & $\cdot$ & $\cdot$ & $1$ & $\cdot$ & $\cdot$ & $3$ & {\small $l\equiv O_1$ (contact), $l_1$}&\\ 
$32$ & $\cdot$ & $\cdot$ & $1$ & $\cdot$ & $\cdot$ & $3$ & {\small $l\equiv O_1$ (osculation)}& {\small $(\phi)$ touch quadric}\\ 
$33$ & $\cdot$ & $1$ & $\cdot$ & $\cdot$ & $\cdot$ & $3$ & {\small $\omega_2\equiv O_1(1)$, $l\equiv O_1$}&\\ 
$34$ & $1$ & $\cdot$ & $\cdot$ & $\cdot$ & $\cdot$ & $3$ & {\small $\omega_3\equiv O_1^2$}&\\ 
$35$ & $1$ & $\cdot$ & $\cdot$ & $\cdot$ & $\cdot$ & $3$ & {\small $l\equiv O_1$ (contact), $l_1\equiv O_1$}& {\small $(\phi)$ touch quadric}\\ 
$36$ & $\cdot$ & $\cdot$ & $2$ & $\cdot$ & $\cdot$ & $3$ & {\small $\omega_2\equiv O_1$, $l\equiv O_1O_2$}&\\  
$37$ & $\cdot$ & $1$ & $1$ & $\cdot$ & $\cdot$ & $3$ & {\small $\omega_2\equiv O_1(1)O_2$, $l\equiv O_1O_2$}&\\ 
$38$ & $\cdot$ & $1$ & $1$ & $\cdot$ & $\cdot$ & $3$ & {\small $l\equiv O_1O_2$, $l_1\equiv O_1(1)$, $l_2\equiv O_1(1)$}&\\ 
$39$ & $1$ & $\cdot$ & $1$ & $\cdot$ & $\cdot$ & $3$ & {\small $l\equiv O_1O_2$ (contact), $l_1\equiv O_1$}& {\small $(\phi)$ touch plane}\\ 
$40$ & $1$ & $1$ & $\cdot$ & $\cdot$ & $\cdot$ & $3$ & {\small $l\equiv O_1O_2(1)$ osculation}& {\small $(\phi)$ touch plane}\\ 
$41$ & $\cdot$ & $\cdot$ & $\cdot$ & $\cdot$ & $1$ & $1$ & {\small $l_1$, $l_2$, $l_3$}&\\ 
$42$ & $\cdot$ & $\cdot$ & $1$ & $\cdot$ & $1$ & $1$ & {\small $\omega_3\equiv O_1$ (rational)}&\\ 
$43$ & $\cdot$ & $\cdot$ & $1$ & $\cdot$ & $1$ & $1$ & {\small $l_1\equiv O_1$, $l_2\equiv O_1$, $l_3$}&\\ 
$44$ & $\cdot$ & $1$ & $\cdot$ & $\cdot$ & $1$ & $1$ & {\small $\omega_3\equiv O_1^2(1)$}&\\ 
$45$ & $\cdot$ & $1$ & $\cdot$ & $\cdot$ & $1$ & $1$ & {\small $\omega_2\equiv O_1(1)$, $l\equiv O_1(1)$}&\\ 
$46$ & $\cdot$ & $1$ & $\cdot$ & $\cdot$ & $1$ & $1$ & {\small $l\equiv O_1(1)$ (contact), $l_1\equiv O_1$}& {\small $(\phi)$ touch quadric}\\
$47$ & $1$ & $\cdot$ & $\cdot$ & $\cdot$ & $1$ & $1$ & {\small $l_1\equiv O_1$, $l_2\equiv O_1$, $l_3\equiv O_1$}&\\  
$48$ & $\cdot$ & $\cdot$ & $2$ & $\cdot$ & $1$ & $1$ & {\small $l\equiv O_1O_2$, $l_1\equiv O_1$, $l_2\equiv O_2$}&\\  
$49$ & $\cdot$ & $1$ & $1$ & $\cdot$ & $1$ & $1$ & {\small $l\equiv O_1(1)O_2$ (contact), $l_2\equiv O_1$}& {\small $(\phi)$ touch plane}\\ 
& & & & & & & &{\small $O_2$ on fixed plane at $O_1$}\\ 
$50$ & $\cdot$ & $\cdot$ & $3$ & $\cdot$ & $1$ & $1$ & {\small $l_1\equiv O_2O_3$, $l_2\equiv O_3O_1$, $l_3\equiv O_1O_2$}&\\ 
$51$ & $\cdot$ & $\cdot$ & $\cdot$ & $1$ & $\cdot$ & $\cdot$ & {\small $\omega_3$ (rational)}&\\ 
$52$ & $\cdot$ & $\cdot$ & $1$ & $1$ & $\cdot$ & $\cdot$ & {\small $\omega_3\equiv O_1^2$}&\\ 
\hline
\end{tabular}
\begin{center}
\text{Cubic Space Transformations of bidegree $(3,6)$}
\end{center}
\end{landscape}

\begin{landscape}
\begin{tabular}{|c|c|c|c|c|c|c|c|l|l|}
\hline
 & & & & & & & & &\\ 
 number & degrees& D.p. of & binode & D. p.'s & point of & point of & ordinary & $F$-curves & Remarks\\
 & & contact & & & osculation & contact & points & &\\ 
 & & & & & & & & &\\ 
\hline
$53$ & $3$--$7$ & $1$ & $\cdot$ & $\cdot$ & $\cdot$ & $\cdot$ & $4$  & {\small $l\equiv O_1$ (contact)} & {\small $(\phi)$ touch quadric}\\ 
$54$ & & $\cdot$ & $\cdot$ & $2$ & $\cdot$ & $\cdot$ & $4$ & {\small $l\equiv O_1O_2$, $l_1$}
& \\ 
$55$ & & $\cdot$ & $1$ & $1$ & $\cdot$ & $\cdot$ & $4$ & {\small $l\equiv O_1O_2$, $l_1\equiv O_1(1)$} &\\ 
$56$ & & $1$ & $\cdot$ & $1$ & $\cdot$ & $\cdot$ & $4$ & {\small $l\equiv O_1O_2$ (contact)} & {\small $(\phi)$ touch plane}\\ 
$57$ & & $\cdot$ & $\cdot$ & $1$ & $\cdot$ & $1$ & $2$ & {\small $\omega_2$} &\\ 
$58$ & & $\cdot$ & $1$ & $\cdot$ & $\cdot$ & $1$ & $2$ & {\small $\omega_2\equiv O_1(1)$} &\\ 
$59$ & & $1$ & $\cdot$ & $\cdot$ & $\cdot$ & $1$ & $2$ & {\small $l_1\equiv O_1$, $l_2\equiv O_1$} &\\ 
$60$ & & $\cdot$ & $\cdot$ & $1$ & $\cdot$ & $2$ & $\cdot$ & {\small $l\equiv O_1$, $l_1$} &\\ 
$61$ & & $\cdot$ & $1$ & $\cdot$ & $\cdot$ & $2$ & $\cdot$ & {\small $l_1\equiv O_1(1)$, $l_2\equiv O_1$} &\\ 
$62$ & & $\cdot$ & $1$ & $\cdot$ & $\cdot$ & $2$ & $\cdot$ & {\small $l\equiv O_1(1)$ (contact)} & {\small $(\phi)$ touch quadric}\\ 
$63$ & & $\cdot$ & $\cdot$ & $2$ & $\cdot$ & $2$ & $\cdot$ & {\small $l\equiv O_1O_2$, $l_1\equiv O_1$} &\\ 
$64$ & & $\cdot$ & $1$ & $1$ & $\cdot$ & $2$ & $\cdot$ & {\small $l_1\equiv O_1(1)O_2$ (contact)} & {\small $(\phi)$ touch plane}\\ 
 & & & & & & & & &{\small $O_2$ on fixed plane at $O_1$}\\ 
$65$ & & $\cdot$ & $\cdot$ & $1$ & $1$ & $\cdot$ & $1$ & {\small $\omega_2\equiv O_1$} &\\ 
$66$ & & $\cdot$ & $1$ & $\cdot$ & $1$ & $\cdot$ & $1$ & {\small $l_1\equiv O_1(1)$, $l_2\equiv O_1(1)$} &\\ 
\hline
$67$ & $3$--$8$ & $\cdot$ & $1$ & $1$ & $\cdot$ & $\cdot$ & $5$  & {\small $l\equiv O_1O_2$} &\\ 
$68$ & & $1$ & $\cdot$ & $\cdot$ & $\cdot$ & $1$ & $3$  & {\small $l\equiv O_1$} &\\ 
$69$ & & $\cdot$ & $1$ & $\cdot$ & $\cdot$ & $2$ & $1$  & {\small $l\equiv O_1$} &\\ 
$70$ & & $\cdot$ & $\cdot$ & $1$ & $1$ & $\cdot$ & $2$  & {\small $l$} &\\ 
$71$ & & $\cdot$ & $1$ & $\cdot$ & $1$ & $\cdot$ & $2$  & {\small $l\equiv O_1(1)$} &\\ 
\hline
$72$ & $3$--$9$ & $1$ & $\cdot$ & $\cdot$ & $\cdot$ & $1$ & $4$ & & \\ 
$73$ & & $\cdot$ & $1$ & $\cdot$ & $\cdot$ & $3$ & $\cdot$  & & \\ 
$74$ & & $\cdot$ & $1$ & $\cdot$ & $1$ & $\cdot$ & $3$  & & \\ 
$75$ & & $\cdot$ & $\cdot$ & $1$ & $1$ & $\cdot$ & $2$  & & {\small $4$-point contact at $O_2$}\\ 
\hline
\end{tabular}
\end{landscape}

\medskip

\bibliographystyle{plain}
\bibliography{biblio}
\nocite{}

\end{document}